\documentclass[preprint]{elsarticle}

\usepackage{lineno}
\usepackage[allcolors=black,hidelinks = true]{hyperref}

\journal{Linear Algebra and its Applications}

\usepackage{lipsum}
\usepackage{amsfonts}
\usepackage{graphicx}
\usepackage{epstopdf}
\usepackage{algorithmic}
\usepackage{amssymb}
\usepackage{amsmath}
\usepackage{mathdots}
\usepackage{mathtools}
\usepackage{amsthm}
\usepackage{graphicx}
\usepackage{color}
\usepackage{todonotes}
\usepackage{amsopn}
\newcommand{\dd}{\!\mathrm{d}}
\newcommand{\E}{\mathrm{e}}

\definecolor{orange}{rgb}{1,0.5,0}

\usepackage{xcolor}
\ifpdf
  \DeclareGraphicsExtensions{.eps,.pdf,.png,.jpg}
\else
  \DeclareGraphicsExtensions{.eps}
\fi

\DeclareMathOperator*{\diag}{diag}
\usepackage{makecell}
\newtheorem{thm}{Theorem}
\newtheorem{lemma}[thm]{Lemma}
\newdefinition{remark}{Remark}
\newdefinition{definition}{Definition}
\newproof{pf}{Proof}
\newtheorem{example}{\textsc{Example}}

\newcommand{\cred}[1]{{\color{red}  #1}}

\begin{document}

\begin{frontmatter}

\title{\cred{Toeplitz Momentary Symbols: definition, results, and limitations in the spectral analysis of Structured Matrices}}

\author[1]{Matthias Bolten}

\author[2]{Sven-Erik Ekstr\"{o}m}

\author[1]{Isabella Furci*}
\cortext[mycorrespondingauthor]{Corresponding author}
\ead{furci@uni-wuppertal.de}

\author[2,3]{Stefano Serra-Capizzano}

\address[1]{{Department of Mathematics and Informatics}, {University of Wuppertal}, {{Wuppertal}, {Germany}}}

\address[2]{{Division of Scientific Computing, Department of Information Technology}, {University of Uppsala}, {{Uppsala}, {Sweden}}}

\address[3]{{Department of Science and high Technology}, {University of Insubria}, {{Como}, {Italy}}}

\begin{abstract}
A powerful tool for analyzing and approximating the singular values and eigenvalues of structured matrices is the theory of Generalized Locally Toeplitz (GLT) sequences. By the GLT theory one can derive a function, called the symbol, which describes the singular value or \cred{the eigenvalue distribution of the sequence, the latter under precise assumptions. However, for small values} of the matrix size of the considered sequence, the approximations may not be \cred{as good as it is desirable, since in the construction of the GLT symbol one disregards small norm and low-rank perturbations. On the other hand, Local Fourier analysis (LFA) can be used to construct polynomial symbols} in a similar manner for discretizations, where the geometric information is present, but the small norm perturbations are retained.

The main focus of this paper is the introduction of the concept of \cred{sequence of ``Toeplitz momentary symbols''}, associated with \cred{a given sequence of truncated} Toeplitz-like matrices. We construct the symbol in the same way as in the GLT theory, but we keep the information of the small norm contributions. The low-rank contributions \cred{are still disregarded, and we give an idea on the reason why this is negligible in certain cases and why it is not in other cases, being aware that in presence of high nonnormality the same low-rank perturbation can produce a dramatic change in the eigenvalue distribution. Moreover, a difference with respect to} the LFA symbols is that \cred{GLT symbols and Toeplitz momentary symbols are more general - just Lebesgue measurable - and} are applicable to a larger class of matrices. We show the applicability of \cred{the approach which leads to higher accuracy in some cases}, when approximating the singular values and eigenvalues of Toeplitz-like matrices using \cred{Toeplitz momentary symbols}, compared with the GLT symbol.
\cred{Finally}, since for many applications and their analysis it is often necessary to consider non-square Toeplitz matrices, we formalize and provide some useful \cred{definitions}, applicable for non-square Toeplitz momentary symbols. 
\end{abstract}


\end{frontmatter}

\nolinenumbers

\section{Introduction}
\label{sec:introduction}

{
In many cases computing the numerical solution of partial
differential equations (PDEs) requires the solution  of structured (sparse) linear systems \cite{ benedusi181, qp, her14}.
Hence, the spectral properties of the related coefficient matrix play a crucial role for designing an efficient and appropriate solver \cite{ADS,sa_SDG, full-Galerkin, MR1990645}. Moreover, the eigenvalues and eigenvectors themselves are of interest in many applications \cite{COTTRELL20065257, hansen}.
}

{
Depending on the linear differential operator and the used method in the discretization process, the associated coefficient matrices can \cred{possess a very nice structure: often the associated matrix sequences belong to the Toeplitz class} or to the more general class of Generalized Locally Toeplitz
(GLT) matrix sequences \cite{ GLT-bookIII, GLT-bookI, GLT-bookII}.
One of the main advantages of belonging to the latter class is that crucial information of the involved matrices can be related to the concept of the symbol, a function which, under certain hypotheses, provides an  asymptotic description of their eigenvalues and singular values.
In the past years the theory of GLT sequences has been largely improved and successfully used for this purpose. However, since the results from the GLT theory \cred{apply only to matrix sequences, it follows that its validity is of asymptotic type. Therefore, for small matrix-sizes $n$,} the approximations may not be as accurate as it is desirable. Indeed, one aspect of the construction of the GLT symbol is that one disregards all parts that are small norm and low-rank perturbations. Consequently, \cred{for moderate size $n$ the spectra of the matrices of interest can significantly differ from those studied by means of the GLT symbol. For instance, in the case of the Schoenmakers-Coffey matrix sequences, the symbol is zero and in fact the eigenvalues cluster at zero, but this is clearly observed only for large sizes of the corresponding concrete matrices and hence this result was not known to people working in the field \cite{SSS}. On the other hand, when employing the GLT approach for solving large linear systems, the results have been very satisfactory, in particular for designing preconditioners for the (preconditioned) Krylov methods and for defining prolongation and restriction operators in multigrid methods, for PDEs and fractional differential equations (FDEs), approximated by local methods such as Finite Differences, Finite Elements, Finite Volumes, Isogeometric Analysis: see \cite{ GLT-bookIII, GLT-ETNA1, GLT-bookI, GLT-bookII} and references therein. The reason of such a success is quite technical and relies on the fact that the spectral approximation has not to be necessarily very accurate: for instance a preconditioning matrix sequence ensuring a clustering of radius $10^{-1}$ is often sufficient for an optimally convergent (preconditioned) Krylov method.}
}

{
Local Fourier analysis (LFA) is another common tool for the analysis of solvers for linear systems arising from the discretization of \cred{PDEs}.  It is predominantly used in the analysis and design of multigrid methods \cite{Oost} and  it contemplates the following two simplifications: we consider only constant coefficient operators and the discrete equation is supposed to be approximated  with an infinite mesh, i.e., the boundary conditions are neglected. Hence, the geometric information is present \cred{and more information is kept in the symbol, since small norm perturbations are retained. However, a strong limitation is that the symbol is of trigonometric polynomial type, while in the GLT approach any Lebesgue measurable function is allowed.}
}

The main aim of the paper is to introduce and exploit the concept of a (singular value and spectral) \cred{``Toeplitz momentary symbols'', associated with a sequence of truncated} Toeplitz-like matrices. Its construction is similar to that of the symbol in the GLT sense, but in practice we keep also the information of the small norm contributions. Even though the low-rank contributions are still disregarded, we give an idea on why this is \cred{negligible, at least in some cases.}

{
In particular, we consider matrix sequences of the form
\begin{equation*}
\{X_n\}_n=\{T_n(f)\}_n+\{N_n\}_n+\{R_n\}_n,
\end{equation*}
where, for every $n$, $T_n(f)$ is a Toeplitz matrix, $N_n$ is a small norm matrix, and $R_n$ is a low-rank matrix. While in the GLT setting an admissible small norm matrix sequence $\{N_n\}_n$ consists \cred{of very general matrices, in our setting we want to consider sequences with a specific structure}. We illustrate the applicability of the momentary symbols in several examples stemming from applications of interest, highlighting its efficacy and higher accuracy, when approximating the singular values and eigenvalues of \cred{truncated} Toeplitz-like matrices, compared with the GLT symbol.
}

{
The structure of the paper is the following. Firstly, in \cred{Subsection} \ref{sec:introduction:background} we fix the notation and introduce the fundamental preliminaries and results as Toeplitz matrices in the multilevel block setting and the concept of (spectral and singular value) asymptotic distributions. \cred{Subsection \ref{sec:introduction:glt} introduces the axioms characterizing the theory of the GLT sequences, while Subsection \ref{sec:introduction:algebras} presents circulant matrices and other common matrix algebras, together with their spectral properties.  Furthermore, in Section \ref{sec:momentary} we define the notion of Toeplitz momentary symbols and we test its applicability in Examples 1-3, with a discussion on its limits and on links with the Local Fourier Analysis (LFA)}. Finally, since for spectral analysis of many problems it is often necessary to consider non-square Toeplitz matrices, in Section \ref{sec:non-square} we formalize and provide some useful definitions, applicable for non-square momentary symbols.
In the conclusive section, we highlight the main findings of the paper and we give an idea of possible extensions and future developments. 
}
\subsection{Background and definitions}
\label{sec:introduction:background}
Let ${f}:G\to\mathbb{C}$ be a function belonging to $L^1(G)$, with $G\subseteq\mathbb R^\ell$, $\ell\ge 1$, measurable set.
We indicate by $\{A_{n}\}_{n}$ the matrix sequence whose elements are the matrices $A_{n}$ of dimension $n \times n$. Let $s,d \in \mathbb{N}$. If $\mathbf{n}=(n_1,n_2,\dots,n_d)$ is a multi-index we indicate by $\{A_{\mathbf{n}}\}_{\mathbf{n}\in\mathbb{N}^d}$, or simply $\{A_{\mathbf{n}}\}_{\mathbf{n}}$, the $d$-level $s\times s$ block matrix sequence whose elements are the matrices $A_\mathbf{n}$ of size $d_\mathbf{n}=d_\mathbf{n}(\mathbf{n},s)=sn_1n_2\cdots n_d$. For simplicity, if not otherwise specified, we report the main background regarding the matrix sequence in the scalar unilevel setting and we will indicate the strategies and references to generalize such results.
\begin{definition}
\label{def:introduction:background:toeplitz}
A {square} Toeplitz matrix $A_n$ of order $n$ is a matrix that has equal {entries} along each diagonal, and is defined by
\begin{linenomath*}
  \begin{align*}
A_n=\left[a_{i-j}\right]_{i,j=1}^{n}=\left[\begin{smallmatrix}
a_0 & a_{-1} & a_{-2} & \cdots & \cdots & a_{1-n}\vphantom{\ddots}\\
a_1 & \ddots & \ddots & \ddots & & \vdots\\
a_2 & \ddots & \ddots & \ddots & \ddots & \vdots\\
\vdots & \ddots & \ddots & \ddots & \ddots & a_{-2}\\
\vdots & & \ddots & \ddots & \ddots & a_{-1}\\
a_{n-1} & \cdots & \cdots & a_2 & a_1 & a_0\vphantom{\ddots}\\
&
\end{smallmatrix}\right].
 \end{align*} 
\end{linenomath*}
A {square} Toeplitz matrix $T_{n}(f) \in \mathbb{C}^{n \times n}$, is associated with a function $f$, called the \textbf{generating function}, belonging to $L^1([-\pi,\pi])$ and periodically extended to the whole real line.
The matrix $T_{n}(f)$ is defined as
\begin{linenomath*}
  \begin{align*}
  T_n(f)=\left[\hat f_{i-j}\right]_{i,j=1}^n,\nonumber
 \end{align*} 
\end{linenomath*}
where
\begin{linenomath*}
  \begin{align}
  \hat{f}_{k}\coloneqq\frac1{2\pi}\int_{-\pi}^{\pi}\!\!f(\theta)\,\E^{-k\mathbf{i} \theta}\dd\theta,\qquad k\in\mathbb Z,\qquad \mathbf{i}^2=-1,\label{eq:introduction:background:fourier}
 \end{align} 
\end{linenomath*}
are the Fourier coefficients of $f$, and
\begin{linenomath*}
  \begin{align}
  f(\theta)=\!\!\sum_{k=-\infty}^{\infty}\!\!\hat{f}_{k}\E^{k\mathbf{i} \theta},\label{eq:introduction:fourierseries}
 \end{align} 
\end{linenomath*}
is the Fourier series of $f$. 
\end{definition}
In the following we can see how to define block Toeplitz matrices $T_n(\mathbf{f})$ starting from matrix-valued function $\mathbf{f}:[-\pi,\pi]\rightarrow \mathbb{C}^{s\times s}$ with $\mathbf{f}\in L^1([-\pi,\pi])$ and, more in general, how define $d$-level block Toeplitz matrices $T_{n}(\mathbf{f})$ starting from  $d$-variate matrix-valued function $\mathbf{f}:[-\pi,\pi]^{d}\rightarrow \mathbb{C}^{s\times s}$
with $\mathbf{f}\in L^1([-\pi,\pi]^d)$. For the block settings we will write the function $\mathbf{f}$ (and corresponding Fourier coefficients) in bold.
In particular we can define the Fourier coefficients of a given function $\mathbf{f}:[-\pi,\pi]^{d}\rightarrow \mathbb{C}^{s\times s}$ as 
\begin{linenomath*}
  \begin{align*}
  \hat{\mathbf{f}}_{\mathbf{k}}\coloneqq
  \frac1{(2\pi)^d}
  \int_{[-\pi,\pi]^d}\mathbf{f}(\boldsymbol{\theta})\E^{-\mathbf{i}\left\langle {\mathbf{k}},\boldsymbol{\theta}\right\rangle}\mathrm{d}\boldsymbol{\theta}\in\mathbb{C}^{s\times s},
  \qquad \mathbf{k}=(k_1,\ldots,k_d)\in\mathbb{Z}^d,\nonumber
 \end{align*} 
\end{linenomath*}
where $\boldsymbol{\theta}=(\theta_1,\ldots,\theta_d)$, $\left\langle \mathbf{k},\boldsymbol{\theta}\right\rangle=\sum_{i=1}^dk_i\theta_i$, and the integrals of matrices are computed elementwise. The associated generating function, from the Fourier coefficients is
\begin{linenomath*}
  \begin{equation}
\label{eq:introduction:matrixvaluedsymbol}
\mathbf{f}(\boldsymbol{\theta})=\sum_{\mathbf{k}}\hat{\mathbf{f}}_{\mathbf{k}}\E^{\mathbf{i}\left\langle {\mathbf{k}},\boldsymbol{\theta}\right\rangle}.
 \end{equation}\end{linenomath*}

\cred{One $\mathbf{n}$th multilevel block Toeplitz matrix associated with $\mathbf{f}$ is the matrix of dimension $d_{\mathbf{n}}$, where $\mathbf{n}=(n_1,\ldots,n_d)$, $d_{\mathbf{n}}=n_1 n_2\cdots n_d s$},  given by
\begin{linenomath*}
  \begin{align*}
T_\mathbf{n}(\mathbf{f})&=
\sum_{\mathbf{e}-\mathbf{n}\le \mathbf{k}\le \mathbf{n}-\mathbf{e}} T_{n_1}(\E^{\mathbf{i}k_1\theta_1})\otimes \cdots \otimes
T_{n_d}(\E^{\mathbf{i}k_d\theta_1})\otimes \hat{\mathbf{f}}_{\mathbf{k}},
\nonumber
 \end{align*} 
\end{linenomath*}
where $\otimes$ denotes the Kronecker product, \cred{$\mathbf{e}$ is the vectors of all ones, and $\mathbf{k}\le \mathbf{q}$ means $k_j\le q_j$ for all $j=1,\ldots,d$}. For a more detailed description and uses of the multi-index notation see \cite{GLT-bookII}.
In the following we introduce the definition of \textit{spectral distribution} in the sense of the eigenvalues and of the singular values for a generic $d$-level matrix sequence $\{A_{\mathbf{n}}\}_{\mathbf{n}}$, and then the notion of GLT algebra.

\begin{definition} \cite{GLT-bookI,GLT-bookII,gsz,TyZ}
\label{def:introduction:background:distribution}
Let $f,{\mathfrak{f}}:G\to\mathbb{C}$ be  measurable functions, defined on a measurable set $G\subset\mathbb{R}^\ell$ with $\ell\ge 1$, $0<\mu_\ell(G)<\infty$.
Let $\mathcal{C}_0(\mathbb{K})$ be the set of continuous functions with compact support over $\mathbb{K}\in \{\mathbb{C}, \mathbb{R}_0^+\}$ and let $\{A_{\mathbf{n}}\}_{\mathbf{n}}$, be a sequence of matrices with eigenvalues $\lambda_j(A_{\mathbf{n}})$, $j=1,\ldots,{d_\mathbf{n}}$ and singular values $\sigma_j(A_{\mathbf{n}})$, $j=1,\ldots,{d_\mathbf{n}}$.
Then,
\begin{itemize}
  \item The matrix sequence  $\{A_{\mathbf{n}}\}_\mathbf{n}$ is \textit{distributed as the pair $(f,G)$ in the sense of the \textbf{singular values}}; we denote this by
    \begin{linenomath*}
  \begin{align*}
      \{A_{\mathbf{n}}\}_{\mathbf{n}}\sim_\sigma(f,G),\nonumber
     \end{align*} 
\end{linenomath*}
    if the following limit relation holds for all $F\in\mathcal{C}_0(\mathbb{R}_0^+)$:
		\begin{linenomath*}
  \begin{align}
		  \lim_{\mathbf{n}\to\infty}\frac{1}{{d_\mathbf{n}}}\sum_{j=1}^{{d_\mathbf{n}}}F(\sigma_j(A_\mathbf{n}))=
		  \frac1{\mu_\ell(G)}\int_G  F({|f(\boldsymbol{\theta})|})\,\dd{\boldsymbol{\theta}}.\label{eq:introduction:background:distribution:sv}
		 \end{align} 
\end{linenomath*}
    The function $f$ is called the \textbf{singular value symbol} which describes the singular value distribution of the matrix sequence $ \{A_{\mathbf{n}}\}_{\mathbf{n}}$.
	\item The matrix sequence $\{A_{\mathbf{n}}\}_{\mathbf{n}}$ is \textit{distributed as the pair $(\mathfrak{f},G)$ in the sense of the \textbf{eigenvalues}}; we denote this by
    \begin{linenomath*}
  \begin{align*}
           \{A_{\mathbf{n}}\}_{\mathbf{n}}\sim_\lambda({\mathfrak{f}},G),\nonumber
     \end{align*} 
\end{linenomath*}
    if the following limit relation holds for all $F\in\mathcal{C}_0(\mathbb{C})$:
    \begin{linenomath*}
  \begin{align}
       \lim_{\mathbf{n}\to\infty}\frac{1}{{d_\mathbf{n}}}\sum_{j=1}^{{d_\mathbf{n}}}F(\lambda_j(A_{\mathbf{n}}))=
      \frac1{\mu_\ell(G)}\int_G \displaystyle  F({\mathfrak{f}(\boldsymbol{\theta})})\,\dd{\boldsymbol{\theta}}.\label{eq:introduction:background:distribution:ev}
     \end{align} 
\end{linenomath*}
    The function $\mathfrak{f}$ is called the \textbf{eigenvalue symbol} which describes the eigenvalue distribution of the matrix sequence $ \{A_{\mathbf{n}}\}_{\mathbf{n}}$.
  \end{itemize}

\end{definition}
\begin{remark}
\label{rem:introduction:background:1}
  If $A_\mathbf{n}$ is Hermitian, then $f=\mathfrak{f}$. 
 For $d=1$, if $f$ (or $\mathfrak{f}$) is smooth enough, an informal interpretation of the limit relation \eqref{eq:introduction:background:distribution:sv} (or \eqref{eq:introduction:background:distribution:ev}) is that when the matrix size of $A_{n}$ is sufficiently large, then the $n$ singular values (or eigenvalues) of $A_{n}$ can, except for possibly $o(n)$ outliers, be approximated by a sampling of $|f(\theta)|$ (or $\mathfrak{f}(\theta)$) on an equispaced grid of the domain~$G$.
 A grid often used to approximate the eigenvalues of a Hermitian matrix $A_n$, $\lambda_j(A_n)\approx f(\theta_{j,n})$, when $f$ is an even function, is
 \begin{linenomath*}
  \begin{align*}
 \theta_{j,n}=\frac{j\pi}{n+1},\quad j=1,\ldots,n.\nonumber
  \end{align*} 
\end{linenomath*}
 The generalization of Definition \ref{def:introduction:background:distribution} and Remark \ref{rem:introduction:background:1} to the block setting can be found in \cite{GLT-bookIII} and in the references therein. 
\end{remark}

\subsection{Theory of Generalized Locally Toeplitz (GLT) sequences}
\label{sec:introduction:glt}

\cred{We list the axioms of the theory of Generalized Locally Toeplitz (GLT) matrix sequences: these axioms represent an equivalent characterization of the original definition of the GLT matrix sequences. While the original definition is quite involved and requires the introduction of several notions (see \cite{GLT-LAA1,GLT-LAA2}), the advantage of the axioms below is that they are operative and emphasize the practical and operational features of the GLT class, see~\cite{, GLT-bookIII, GLT-ETNA1,GLT-bookI, GLT-bookII} for further details. 
We choose to report the axioms in the general multilevel and block setting. Nevertheless, we will specify section by section what type of matrix sequences we are considering.} In this paper we restrict our attention to the constant coefficient case. However, possible generalizations for the variable coefficient setting can be treated and will be the object of future research.

\begin{description}
  \item[GLT1] \cred{Each GLT sequence has a GLT symbol $\mathbf{f}(\boldsymbol{\theta})$ with $\boldsymbol{\theta}\in [-\pi,\pi]^d$, that is 
  $\{A_\mathbf{n}\}_\mathbf{n}\sim_{\textsc{glt}} \mathbf{f}(\boldsymbol{\theta})$. The GLT symbol is also  a singular value symbol,  according to the second item in Definition \ref{def:introduction:background:distribution} with $\ell=d$. 
  If the sequence is Hermitian, then the distribution also holds in the eigenvalue sense.}
  	
  \item[GLT2] The set of GLT sequences form a $*$-algebra, i.e., it is closed under linear combinations, products, inversion (whenever the symbol is singular, at most, in a set of zero Lebesgue measure), and conjugation.
  Hence, \cred{as a particular case, the GLT matrix sequence obtained via algebraic operations  of a finite set of GLT matrix sequences has symbol given  by performing the same algebraic manipulations of the symbols of the considered GLT matrix sequences.}
	
  \item[GLT3] Every Toeplitz sequence $\{T_n(\mathbf{f})\}_n$ generated by a function $\mathbf{f}=\mathbf{f}(\boldsymbol{\theta})$ belonging to $L^1([-\pi,\pi]^d)$ is a GLT sequence and its GLT symbol 	is $\mathbf{f}$. \cred{Each diagonal sampling sequence $\{D_n(\mathbf{a})\}_n$ with $\mathbf{a}$ Riemann-integrable over $[0,1]^d$ is a GLT sequence and its GLT symbol is $\mathbf{a}$.}
	
  \item[GLT4] Every sequence which is distributed as the constant zero in the singular value sense is a GLT sequence with
	symbol~$0$. In particular:
	\begin{itemize}
		\item every sequence in which the rank divided by the size tends to zero, as the matrix size tends to infinity;
		\item every sequence in which the trace-norm (i.e., sum of the singular values) divided by the size tends to zero, as the matrix size tends to infinity.
	\end{itemize}	 
\end{description}

\subsection{Eigenvalues and eigenvectors for common matrix algebras}
\label{sec:introduction:algebras}
We here introduce notation regarding a few common matrix algebras,  to justify the choice of a specific sampling grid in subsequent sections.
We consider particular matrix algebras $ {\tau_{\varepsilon,\varphi}}$ and the circulant algebra.

\cred{For real parameters $\varepsilon,\varphi$,} the $ {\tau_{\varepsilon,\varphi}}$-algebras are special cases, first introduced in~\cite{bozzo}, of the wider class of $\tau$-algebras, see \cite{bozzo} and references therein. {A matrix in the $ {\tau_{\varepsilon,\varphi}}$-algebra is a polynomial of the generator }

\begin{linenomath*}
  \begin{align*}
 {T_{n,\varepsilon,\varphi}}=
\left[\begin{smallmatrix}
\varepsilon &1\\
1&0&1\\
&\ddots&\ddots&\ddots\\
&&1&0&1\\
&&&1&\varphi 
\end{smallmatrix}\right]
.\nonumber
\end{align*}
\end{linenomath*}

{ Here we restrict the analysis to the case where an element in the algebra  $ {\tau_{\varepsilon,\varphi}}$ is a matrix, denoted  {$T_{n,\varepsilon,\varphi}(f)$}, generated by a function $f$ of the form $f(\theta)=\hat{f}_0+2\hat{f}_1\cos\theta$ that is a matrix  of the form}

\begin{linenomath*}
  \begin{align}
 {T_{n,\varepsilon,\varphi}(f)}=
\left[\begin{smallmatrix}
\hat{f}_0+\varepsilon \hat{f}_1&\hat{f}_1\\
\hat{f}_1&\hat{f}_0&\hat{f}_1\\
&\ddots&\ddots&\ddots\\
&&\hat{f}_1&\hat{f}_0&\hat{f}_1\\
&&&\hat{f}_1&\hat{f}_0+\varphi \hat{f}_1
\end{smallmatrix}
\right]
=
T_n(f)+
\left[\begin{smallmatrix}
\varepsilon \hat{f}_1\\
\\
\\
&&&\varphi \hat{f}_1
\end{smallmatrix}
\right]
,\nonumber
 \end{align} 
\end{linenomath*}\label{eq:tau_decomposition}
where {$|\varepsilon|,|\varphi|\le 1$}. {For discussions on the case $|\varepsilon|,|\varphi|> 1$, see \cite{Jozefiak}}.
Note that, for $|\varepsilon|, |\varphi|<1$, \begin{linenomath*}
  \begin{align*}
{ {T_{n,\varepsilon,\varphi}(f)}=\mathbb{Q}_nD_n(f)\mathbb{Q}_n^{\textsc{t}},}
 \end{align*} 
\end{linenomath*}
where $D_n$ is a diagonal matrix and $\mathbb{Q}_n$ is a real-valued unitary matrix ($\mathbb{Q}_n\mathbb{Q}_n^{\textsc{t}}=\mathbb{I}_n$) depending on $(\varepsilon,\varphi)$. 
The entries on the diagonal of $D_n(f)$ are the eigenvalues of  {$T_{n,\varepsilon,\varphi}(f)$}, which are explicitly given by the sampling of $f$ on a grid $\theta_{j,n}^{(\varepsilon,\varphi)}$.
That is,
\begin{linenomath*}
  \begin{align}
\lambda_j( {T_{n,\varepsilon,\varphi}(f)})&=f(\theta^{(\varepsilon,\varphi)}_{j,n}),\quad j=1,\ldots,n,\nonumber\\
D_n(f)&=\diag(f(\theta^{(\varepsilon,\varphi)}_{j,n})),\quad j=1,\ldots,n.\nonumber
 \end{align} 
\end{linenomath*}
The matrix $\mathbb{Q}_n$ {, which depends on $(\varepsilon,\varphi)$, is often referred to as a discrete sine (or cosine) transform (typically denoted by, for example, \texttt{dst-1}, \texttt{dct-1}; see, e.g., \cite[Appendix 1]{CeccheriniSilberstein2008}). We here define it as,}
\begin{linenomath*}
  \begin{align}
(\mathbb{Q}_n)_{i,j}&=\sqrt{2h}\sin(\Theta^{(\varepsilon,\varphi)}_{i,j,n}),\quad i,j=1,\ldots,n,\nonumber
 \end{align} 
\end{linenomath*}
where $\Theta^{(\varepsilon,\varphi)}_{i,j,n}$ is a grid depending on $(\varepsilon,\varphi)$ and $h$ is the denominator of the corresponding grid $\theta^{(\varepsilon,\varphi)}_{j,n}$ (e.g., for $\varepsilon=\varphi=-1,$ $\theta^{(\varepsilon,\varphi)}_{j,n}=j\pi/n$, then, $h=1/n$).

In Table~\ref{tbl:taugrids} we present the proper grids $\theta^{(\varepsilon,\varphi)}_{j,n}$ and $\Theta_{i,j,n}^{(\varepsilon,\varphi)}$ to give the exact eigenvalues and eigenvectors respectively for $\varepsilon, \varphi \in\{-1,0,1\}$. Note that for  {$T_{n,-1,-1}$} matrices the $n$th eigenvector (column $n$ of $\mathbb{Q}_n$) and for  {$T_{n,1,1}$} matrices the first eigenvector (column one of $\mathbb{Q}_n$) have to be normalized by $1/\sqrt{2}$. 

\begin{table}[!ht]
\centering
\caption{Grids for  {$\tau_{\varepsilon,\varphi}$}-algebras,  $\varepsilon,\varphi\in\{-1,0,1\}$; $\theta_{j,n}^{(\varepsilon,\varphi)}$ and $\Theta_{j,n}^{(\varepsilon,\varphi)}$ are the grids used to compute the eigenvalues and eigenvectors, respectively. {The standard naming convention (\texttt{dst-*} and \texttt{dct-*}) in parenthesis; see, e.g., \cite[Appendix 1]{CeccheriniSilberstein2008}.}}
\label{tbl:taugrids}
\begin{tabular}{|r|ccc||c|}
\hline
&&$\theta_{j,n}^{(\varepsilon,\varphi)}$&&$\Theta_{i,j,n}^{(\varepsilon,\varphi)}$\\[0.2em]
\hline
\diaghead{\theadfont MMMMM}{ $\varepsilon$ }{ $\varphi$ }&
\thead{-1}&\thead{0}&\thead{1}&\thead{-1, 0, 1}\\[0.2em]
\hline
\thead{-1}&\makecell{{{\tiny (\texttt{dst-2})}}\\$\frac{j\pi}{n}$}&\makecell{{{\tiny (\texttt{dst-6})}}\\$\frac{j\pi}{n+1/2}$}&\makecell{{{\tiny (\texttt{dst-4})}}\\$\frac{(j-1/2)\pi}{n}$}&$(i-1/2)\theta_{j,n}^{(\varepsilon,\varphi)}$\\[0.2em]
\thead{0}&\makecell{{{\tiny (\texttt{dst-5})}}\\$\frac{j\pi}{n+1/2}$}&\makecell{{{\tiny (\texttt{dst-1})}}\\$\frac{j\pi}{n+1}$}&\makecell{{{\tiny (\texttt{dst-7})}}\\$\frac{(j-1/2)\pi}{n+1/2}$}&$i\theta_{j,n}^{(\varepsilon,\varphi)}$\\[0.2em]
\thead{1}&\makecell{{{\tiny (\texttt{dct-4})}}\\$\frac{(j-1/2)\pi}{n}$}&\makecell{{{\tiny (\texttt{dct-8})}}\\$\frac{(j-1/2)\pi}{n+1/2}$}&\makecell{{{\tiny (\texttt{dct-2})}}\\$\frac{(j-1)\pi}{n}$}&$(i-1/2)\theta_{j,n}^{(\varepsilon,\varphi)}+\frac{\pi}{2}$\\[0.2em]
\hline
\end{tabular}
\end{table}

Since all grids $\theta_{j,n}^{(\varepsilon,\varphi)}$ associated with  {$\tau_{\varepsilon,\varphi}$}-algebras where $\varepsilon,\varphi\in\{-1,0,1\}$ are uniformly spaced grids, we know that
\begin{alignat*}{7}
\theta_{j,n}^{(1,1)}&<\theta_{j,n}^{(0,1)}&&=\theta_{j,n}^{(1,0)}\nonumber\\
&&&<\theta_{j,n}^{(-1,1)}&&=\theta_{j,n}^{(1,-1)}\nonumber\\
&&&&&<\theta_{j,n}^{(0,0)}\nonumber\\
&&&&&<\theta_{j,n}^{(-1,0)}&&=\theta_{j,n}^{(0,-1)}\nonumber\\
&&&&&&&<\theta_{j,n}^{(-1,-1)},\qquad\qquad \forall j=1,\ldots,n.\label{eq:gridcomparison}
\end{alignat*}

Moreover, for a monotone $f$, and using~\cite[Theorem 2.12]{GLT-bookI} it is possible also give bounds for eigenvalues of matrices belonging to  {$\tau_{\varepsilon,\varphi}$}-algebras where $\varepsilon,\varphi\in [-1,1]$ (and are typically not equispaced) using the known eigenvalues for $\varepsilon,\varphi\in \{-1,0,1\}$.

We now describe matrices $C_n(f)$ belonging to the \textbf{circulant algebra}.
Let the Fourier coefficients of a given {function ${f}\in L^1([-\pi,\pi])$ be defined as in \eqref{eq:introduction:background:fourier}.}
 The ${n}$th circulant matrix generated by ${f}$ is the matrix of dimension $n$ given by

{\begin{definition}\label{def:Cir}
Let the Fourier coefficients of a given function ${f}\in L^1([-\pi,\pi])$ be defined as in formula (\ref{eq:introduction:background:fourier}).
Then, we can define the ${n}$th
circulant matrix $C_{n}(f)$ associated with $f$, which is the square matrix of order $n$ given by:
 \begin{linenomath*}
  \begin{equation}
 C_{n}(f)=\!\!\!\!\!\!\sum_{j=-(n-1)}^{n-1}\!\!\!\!\!\!\hat{f}_{j}Z_{n}^{j}=\mathbb{F}_{n}  D_{n}(f) \mathbb{F}_{n}^{\textsc{h}},\label{eq:introduction:background:circulant:schur}
  \end{equation}\end{linenomath*} where $Z_{n}$ is the $n \times n$ matrix defined by
\begin{linenomath*}
  \begin{align}
\left(Z_{n}\right)_{ij}=\begin{cases}
1,&\text{if }\mathrm{mod}(i-j,n)=1,\\
0,&\text{otherwise}.
\end{cases}\nonumber
 \end{align} 
\end{linenomath*}
Moreover,
\begin{linenomath*}
  \begin{equation}\label{eig-circ}
  D_{n}(f)=\diag\left(s_n(f)(\theta_{j,n}^c)\right),\quad j=1,\ldots,n,
 \end{equation}\end{linenomath*}
where
\begin{linenomath*}
  \begin{align}
   \theta_{j,n}^c=\frac{(j-1)2\pi}{n},\quad j=1,\ldots,n,\label{eq:introduction:background:circulant:grid-circ}
  \end{align} 
\end{linenomath*}
and $s_{n}(f)(\theta)$ is the $ n$th Fourier sum of $f$ given by
\begin{linenomath*}
  \begin{equation}\label{fourier-sum}
s_{n}(f)({\theta})= \sum_{k=1-n}^{n-1}  \hat{f}_{k}
\E^{k\mathbf{i}\theta}.
 \end{equation}\end{linenomath*}
The matrix $\mathbb{F}_n$ is the so called Fourier matrix of order $n$, given by
 \begin{linenomath*}
  \begin{align}
 (\mathbb{F}_{n})_{i,j}=\frac{1}{\sqrt{n}} \E^{\mathbf{i}(i-1)\theta_{j,n}^c}, \quad i,j=1,\ldots,n.
  \end{align} 
\end{linenomath*}

\end{definition}
Then, the columns of the Fourier matrix $\mathbb{F}_n$  are the eigenvectors of $C_{n}(f)$.
The proof of the second equality in (\ref{eq:introduction:background:circulant:schur}) can be found in \cite[Theorem 6.4]{GLT-bookI}.
 
\cred{One must be aware that $C_n(f)$ is a good approximation of $T_n(f)$ only when $S_n(f)(\theta)$ converges to $f(\theta)$ in infinity norm, and this is highly nontrivial. In fact the latter is guaranteed only for continuous $2\pi$-periodic functions belonging to the Dini-Lipschitz class, while there exist counterexamples when this condition is violated (see \cite{estatico-serra} and references therein, as the classical book by Zygmund \cite{Zygmund}). In general, an approximation ensuring that the matrix sequence $\{T_n(f)-\tilde C_n(f)\}$ is zero distributed can be obtained for $f\in L^1([-\pi,\pi])$ and for $\tilde C_n(f)$ being the Frobenius optimal approximation of $T_n(f)$ in the circulant algebra (see \cite{Skoro1,Skoro2} and references there reported).
However, when $f$ is smooth, the set of smooth functions being a tiny subset of the Dini-Lipschitz class, the approximation produced by $C_n(f)$ is much more precise than that given by the circulant Frobenius optimal approximation, see \cite{Slinear}.}
 
In addition, if ${f}$ is a trigonometric polynomial of fixed degree less than $n$,   the entries of $D_n(f)$ are the eigenvalues of $C_n(f)$, explicitly given by sampling the generating function $f$ {using} the grid $\theta_{j,n}^c$,
 \begin{linenomath*}
  \begin{align}
  \lambda_j(C_n(f))&=f\left(\theta_{j,n}^c\right),\quad j=1,\ldots,n,\nonumber\\
   D_{n}(f)&=\diag\left(f\left(\theta_{j,n}^c\right)\right),\quad j=1,\ldots,n. \label{eq:introduction:background:circulant:eig-circSEE}
  \end{align} 
\end{linenomath*}

 When $C_n(f)$ is real symmetric, alternatives to the standard Fourier matrix decomposition in \eqref{eq:introduction:background:circulant:schur}, can be constructed using the discrete sine transform, as for the  {$\tau_{\varepsilon,\varphi}$}-algebras. This is due to the fact that the real and imaginary parts of the Fourier matrix are eigenvectors too. 
 Hence, a real-valued $\mathbb{Q}_n$ such that
 \begin{linenomath*}
  \begin{align}
 C_n(f)&=\mathbb{Q}_nD_n(f)\mathbb{Q}_n^\textsc{t},\nonumber
  \end{align} 
\end{linenomath*}
can, for example, be defined as
\begin{linenomath*}
  \begin{align}
(\mathbb{Q}_{n})_{i,j}&=\sqrt{2h}\sin(\Theta_{i,j,n}^{c}),\qquad \Theta_{i,j,n}^{c}=
\begin{cases}
i\theta_{j,n}^c+\frac{\pi}{2},&j=1,\ldots,\left\lfloor\frac{n+2}{2}\right\rfloor,\\
i\theta_{j,n}^c,&j=\left\lfloor\frac{n+2}{2}\right\rfloor+1,\ldots,n,\\
\end{cases}\nonumber
 \end{align} 
\end{linenomath*} 
where $h=1/n$. Note that the elements of column $j=1$ of $\mathbb{Q}_n$ have to be normalized by $1/\sqrt{2}$. For $n$ even also the elements of column $j=n/2+1$ has to be normalized by $1/\sqrt{2}$.}

The generalization for a $d$ variate $s\times s$ matrix-valued $\mathbf{f}$ via a tensor product argument  of the decompositions (\ref{eq:tau_decomposition}) and (\ref{eq:introduction:background:circulant:schur}), can be obtained easily. \cite{ GLT-bookIII, qp}

\section{Toeplitz momentary symbols: definition, results, and limitations}
\label{sec:momentary}

Consider sequences of unilevel matrices $X_n\in\mathbb{C}^{n\times n}$ of the form
\begin{linenomath*}
  \begin{align}
\{X_n\}_n=\{T_n(f)\}_n+\{N_n\}_n+\{R_n\}_n,
\label{eq:sequence}
 \end{align} 
\end{linenomath*}
where, for every n, $T_n(f)$ is a Toeplitz matrix generated by $f$, $N_n$ is a small norm matrix, and $R_n$ is a low-rank matrix, \cred{in the sense that its rank divided by the size tends to zero as the matrix size tends to infinity.}

\cred{For clarity in this section we consider the the sequences only in the unilevel, scalar form (\ref{eq:sequence}). However, the following theory holds also for more general sequences. It can be easily generalized to circulant sequences, where instead of $T_n(f)$ we consider circulant matrices $C_n(f)$, as in Definition \ref{def:Cir}, but with the restriction to $f$ trigonometric polynomial or by considering the Frobenius optimal approximation with no restriction on the symbol (see the discussion in \cite[Remark 0.1]{GLT-LAA2}. Moreover, we can also consider sequences generated by a multivariate and matrix-valued generating function $\mathbf{f}$, $\{T_{\mathbf{n}}(\mathbf{f})\}_\mathbf{n}$.  Also, algebraic combinations (addition, multiplication, and inversion) of different GLT matrix sequences $\{X_n\}_n$ are valid and this is due to the $*$-algebra nature of GLT matrix sequences.}
\cred{Finally, we highlight that future attention will be given to the matrix sequences involving also diagonal sampling matrices $D_n(a)$, $a : [0, 1]^d \rightarrow \mathbb{C}$ (see {\textbf{GLT3}}), that will permit us to treat also variable coefficient matrix sequences: in the current paper we restrict our attention to the GLT matrix sequences generated only by Toeplitz matrix sequences with $L^1$ generating functions and zero-distributed matrix sequences and this means that we are considering a closed $*$-subalgebra of the general GLT class. 

As described in Section \ref{sec:introduction:background} the generating function for a sequence of Toeplitz matrices $\{T_n(f)\}_n$ is $f$.
The matrix sequences $\{N_n\}_n$ and  $\{R_n\}_n$  are small norm and low-rank matrix sequences in the sense described by the items in {\textbf{GLT4}}.

With the proposed notation, we introduce the notion of  \textit{Toeplitz momentary symbols}. 

 \begin{definition}[Toeplitz momentary symbols]
\label{def:momentarysymbols}
Let $\{X_n\}_n$ be a matrix sequence and assume that there exist matrix sequences $\{A_n^{(j)}\}_n$, scalar sequences $c_n^{(j)}$, $j=0,\ldots,t$,
and measurable functions $f_j$ defined over $[-\pi,\pi]$, $t$ nonnegative integer independent  of $n$, such that
\begin{eqnarray} \nonumber
\left\{ \frac{A_n^{(j)}}{ c_n^{(j)}}\right\}_n & = & T_n(f_j), \\  \label{rel:coeff}
c_n^{(0)}=1, & & c_n^{(s)}=o(c_n^{(r)}), \ \ t\ge s>r, \\ 
\{X_n\}_n & = & \{A_n^{(0)}\}_n + \sum_{j=1}^t \{A_n^{(j)}\}_n.
\end{eqnarray}
Then, by a slight abuse of notation,
\begin{equation}\label{eq:T_momentary_1D}
f_n=f_0+ \sum_{j=1}^t c_n^{(j)} f_j
\end{equation}

is defined as the Toeplitz momentary symbol for $X_n$ and $\{f_n\}$ is the sequence of Toeplitz momentary symbols for the matrix sequence $\{X_n\}_n$.
\end{definition}

According to Subsection \ref{sec:introduction:background}, the Toeplitz momentary symbols could be matrix-valued with a number of variables equal to $d$ and domain $[-\pi,\pi]^d$, if the basic matrix-sequences appearing in Definition \ref{def:momentarysymbols} are, up to proper scaling, multilevel Toeplitz matrix sequences with matrix-valued generating functions.
For example in the scalar $d$-variate setting relation (\ref{eq:T_momentary_1D}) takes the form
\[f_{\textbf{n}}= \sum_{\textbf{j}=\textbf{0}}^\textbf{t} c_\textbf{n}^{(\textbf{j})} f_\textbf{j},\]
which is a plain multivariate (possibly block) version of (\ref{eq:T_momentary_1D}).}

\cred{As expected there are links with the notion of Toeplitz generating function and with the GLT theory, as reported in the next result. Its proof is trivial and relies essentially on the structure of the considered matrix sequences and on the assumption in (\ref{rel:coeff}).

\begin{thm}\label{moment-vs-glt-t}
Assume that the matrix sequence $\{X_n\}_n$ satisfies the requirements in Definition \ref{def:momentarysymbols}. Then $\{X_n\}_n$ is a GLT matrix sequence and the generating function $f_0$ of the main term $A_n^{(j)}=T_n(f_0)$ is the GLT symbol of $\{X_n\}_n$, that is, $\{X_n\}_n \sim_{\textsc{glt}}  f_0$.
Furthermore $\lim_{n\to \infty} f_n=f_0$ uniformly on the definition domain.
\end{thm}
}

\cred{Definition \ref{def:momentarysymbols} is quite general and in our examples we require some restrictions. In the following we will focus our attention to the case of three terms, i.e. $t=2$, as it happens in the approximation of second order differential operators. As already mentioned our examples will belong to this more specific framework. 

More in detail, we take into considerations the following three components. }
\begin{enumerate}
\item {$g^{(1)}(n)=1$ for all $n$:} The matrix $g^{(1)}(n)T_n(f_1)=T_n(f_1)$ is the Toeplitz matrix generated by $f$;\label{it:g1}
\item {$g^{(2)}(n)\to0$:} The matrix $g^{(2)}(n)T_n(f_2)=N_n(f_2)$ is a small norm matrix, such that {$\|N_n(f_2)\|\to0$} as $n\to\infty$;\label{it:g2}
\item {$|g^{(0)}(n)|\to\infty:$} The matrix $g^{(0)}(n)T_n(f_0)=L_n(f_0)$ is a diverging matrix. ($L_n$ denoting ``large-norm''); \label{it:g3}
\item if we define  $\{\hat X_n\}_n=\{g^{(1)}(n)T_n(f_1)\}_n+\{g^{(2)}(n)T_n(f_2)\}_n+\{g^{(0)}(n)T_n(f_0)\}_n$, then $\left\{\hat X_n\over g^{(0)}(n)\right\}_n$
is a matrix sequence satisfying Definition \ref{def:momentarysymbols}, while $\{\hat X_n\}_n$ is its non-normalized version (as it is considered in the LFA setting).
\end{enumerate}

In the multivariate case, where $\mathbf{n}=(n_1,\ldots,n_d)$ we denote the function by \cred{$g^{(j)}(\mathbf{n})$, $j=0,1,2$.}

\cred{With the previous notations, given the matrix sequence 
\begin{linenomath*}
  \begin{align*}
  \{X_n\}_n=\{g^{(1)}(n)T_n(f_1)\}_n+\{g^{(2)}(n)T_n(f_2)\}_n+\{g^{(0)}(n)T_n(f_0)\}_n+\{R_n\}_n,
   \end{align*} 
\end{linenomath*}
where $g^{(i)}(n)$ can be of the form described in items \ref{it:g1}-\ref{it:g3}  and $R_n$ is a low-rank matrix, that is $\frac{{\rm rank}(R_n)}{n}\rightarrow 0$, the nonscaled Toeplitz momentary symbol is defined as
\begin{linenomath*}
  \begin{align}
  f_n(\theta)=\sum_{i=0}^{2}g^{(i)}({n})f_i(\theta),\label{eq:def:momentarysymbol}
 \end{align} 
\end{linenomath*}
and of course $f_0(\theta)$ is the GLT symbol of the matrix sequence $\left\{\frac{X_n}{g^{(0)}(n)}\right\}_n$.
}

\cred{We now illustrate specific examples in which the new notion cannot help, at least when the eigenvalues are considered.}

\begin{remark}
\label{rem:wild eigs}
\cred{The remark is composed by two specific examples showing the different stability of eigenvalues and singular values, under a perturbation of minimal rank one and as small as we want in spectral norm.}
\begin{description}
\item[Case 1: eigenvalue distribution and Toeplitz momentary symbols $\quad$]

\cred{ Consider the matrices $T_n(e^{\mathbf{i} \theta})$ and $X_n=T_n(e^{\mathbf{i} \theta})+ \alpha e_1 e_n^T$ with $\alpha\neq 0$. By direct inspection $\{\alpha e_1 e_n^T \}_n\sim_{\sigma,\lambda} 0$ and hence it is a GLT matrix sequence with zero symbol, independently of the parameter $\alpha$. If we look at the GLT momentary symbols, then they coincide with the GLT symbol for both $\{T_n(e^{\mathbf{i} \theta})\}_n$ and $\{X_n\}_n$: however while in the first case, the eigenvalues are all equal to zero and hence $\{\alpha T_n(e^{\mathbf{i} \theta}) \}_n\sim_\lambda 0$, in the second case they distribute asymptotically as the GLT symbol $e^{\mathbf{i} \theta}$ (which is also the GLT momentary symbol for any $n$). This shows that in the nonnormal setting the distribution function (if it exists) can be discontinuous with respect to any reasonable norm of the matrix sequence, since the modulus of the parameter $\alpha$ is allowed to be as small as we want.}
\item[Case 2: singular value distribution and Toeplitz momentary symbols]
\cred{
Take the same example as before. Again $\{\alpha e_1 e_n^T \}_n$ is a GLT matrix sequence with zero symbol, independently of the parameter $\alpha$, and hence we deduce that both $\{T_n(e^{\mathbf{i} \theta})\}_n$ and $\{X_n\}_n$ share the same GLT symbol $e^{\mathbf{i} \theta}$ (which is also the momentary symbol for any $n$). However, from the viewpoint of the singular values no dramatic change is observed and the GLT symbol describes well the singular values of both the matrix sequences. In fact, due to the interlacing results for singular values, from the GLT theory we know that zero-distributed 
matrix sequences do not change the singular value distribution which is continuous and stable with respect to the entries of the matrix sequence.}
\end{description}
\cred{As already mentioned, in this setting, it must be emphasized that the asymptotic eigenvalue distribution is discontinuous with respect to the standard norms or metrics widely considered in the context of matrix sequences, as the approximating class of sequences (a.c.s.) metric.}
\end{remark}

\begin{remark}
\label{rem:LFA}
\cred{
Local Fourier Analysis (LFA) is a common tool for the analysis of solvers for linear systems arising from the discretization of PDEs. However, two simplifications are made: i) Only constant coefficient operators are considered and ii) the discrete equation is considered on an infinite mesh, i.e., the boundary conditions are neglected. For a given grid spacing $\mathbf{h} \in \mathbb{R}^d$ the infinite grid is given by
\[
\Omega_\mathbf{h} := \{ \mathbf{h} \cdot \mathbf{k} : \mathbf{k} \in \mathbb{Z}^d\}.
\]
Discretizing the PDE using, e.g., finite differences yields a stencil representation of the differential operator. Often this representation includes the grid spacing $\mathbf{h}$. As a consequence, in general the symbol $f_{\mathbf{h}}$ tends to infinity for $\mathbf{h} \rightarrow 0$, as in (\ref{eq:def:momentarysymbol}). Usually, the grid spacing depends on the system size, thus from a GLT-viewpoint the matrix is a ``large-norm'' matrix that is not covered by the GLT theory, even if a simple scaling allows to employ again all the GLT tools. Nevertheless, the inclusion of the $\mathbf{h}$ allows for, e.g., the analysis of discretizations of PDEs involving first and second order derivatives.}

\cred{
The discrete operator on the infinite grid can be represented as an infinite matrix. In LFA the approximations to the eigenvalues of this operator are obtained by evaluating the symbol of the operator at equispaced points, thus by the eigenvalues of a circulant matrix with the same symbol. For non-normal matrices this yields a huge deviation from the true eigenvalues, when small matrices are considered. To overcome this limitation, semi-algebraic analysis techniques have been developed~\cite{MR3367826}. An introduction to LFA and its use in multigrid methods can be found in~\cite{MR2108045}.}
\end{remark}

We here illustrate the applicability of the momentary symbols introduced in \cred{Definition}~\ref{def:momentarysymbols}.
\begin{example}\label{exmp:1}
For the second order finite difference discretization of the problem
\begin{linenomath*}
  \begin{align}
  \begin{cases}
u''(x)+u(x)=f(x),&x\in(0,1), \\
u(x)=0,&x=0,\\
u'(x)=0,&x=1,
\end{cases}\label{eq:exmp1:problem}
 \end{align} 
\end{linenomath*}
we have $X_n\mathbf{u}_n=\mathbf{f}_n$, where,
\begin{linenomath*}
  \begin{align}
  X_n=\underbrace{\frac{1}{h^2}\left[
\begin{smallmatrix}
2&-1\\
-1&2&-1\\
&\ddots&\ddots&\ddots\\
&&-1&2&-1\\
&&&-1&2
\end{smallmatrix}\right]}_{L_n(f_0)}
+\underbrace{\mathbb{I}_n}_{T_n(f_1)}
+\frac{1}{h^2}\underbrace{
\left[\begin{smallmatrix}
\\
\\
\phantom{\ddots}\\
\\
&&&&&-1
\end{smallmatrix}\right]}_{R_n}
,\nonumber
 \end{align} 
\end{linenomath*}
where $h=(n+1)^{-1}$.
By notation established above we have,
\begin{linenomath*}
  \begin{align}
X_n&=L_n(f_0)+T_n(f_1)+h^{-2}R_n\nonumber\\
&=g^{(0)}(n)T_n(f_0)+g^{(1)}(n)T_n(f_{{1}})+h^{-2}R_n,\nonumber
 \end{align} 
\end{linenomath*}
where
\begin{linenomath*}
  \begin{align}
g^{(0)}(n)&=h^{-2},
&&f_0(\theta)=2-2\cos\theta,\nonumber\\
g^{(1)}(n)&=1,
&&f_1(\theta)=1.\nonumber
 \end{align} 
\end{linenomath*}
 However, using the definition of Toeplitz momentary symbol, we deduce that $\{X_n\}_n$ has nonscaled Toeplitz momentary symbol given by 
 \begin{linenomath*}
  \begin{align}
f_n(\theta)&=g^{(0)}(n)f_0(\theta)+g^{(1)}(n)f_1(\theta)\nonumber\\
&=h^{-2}(2-2\cos\theta)+1.\nonumber
 \end{align} 
\end{linenomath*}

The standard GLT approch cannot be used for the sequence $\{X_n\}_n$ as it is defined, since the first term diverges. 
If we want to {be able to} construct a GLT symbol, {for instance to analyze a solver for a linear system}, we should normalize the matrix $X_n$ {by multiplication with} $h^2$ obtaining
\begin{linenomath*}
  \begin{align}
  h^2X_n&=\underbrace{
\left[\begin{smallmatrix}
2&-1\\
-1&2&-1\\
&\ddots&\ddots&\ddots\\
&&-1&2&-1\\
&&&-1&2
\end{smallmatrix}\right]}_{T_n(f_0)}
+\underbrace{h^2\mathbb{I}_n}_{N_n(f_1)}
+\underbrace{\left[\begin{smallmatrix}
\\
\\
\phantom{\ddots}\\
\\
&&&&&-1
\end{smallmatrix}\right]}_{R_n}
\nonumber\\
&=\left[\begin{smallmatrix}
2+h^2&-1\\
-1&2+h^2&-1\\
&\ddots&\ddots&\ddots\\
&&-1&2+h^2&-1\\
&&&-1&1+h^2
\end{smallmatrix}\right].\label{eq:exmp1:matrix}
 \end{align} 
\end{linenomath*}
The matrix $h^2X_n$ can be written as
\begin{linenomath*}
  \begin{align}
h^2X_n&=T_n(f_0)+N_n(f_1)+R_n\nonumber\\
&=g^{(0)}(n)T_n(f_0)+g^{(1)}(n)T_n(f_1)+R_n,\nonumber
 \end{align} 
\end{linenomath*}
where
\begin{linenomath*}
  \begin{align}
g^{(0)}(n)&=1,
&&f_0(\theta)=2-2\cos\theta ,\nonumber\\
g^{(1)}(n)&=h^2,
&&f_1(\theta)=1.\nonumber
 \end{align} 
\end{linenomath*}
Since $\{h^2X_n\}_n$ is Hermitian and both the sequences $\{N_n\}_n$ and $\{R_n\}_n$ are zero-distributed (by \textbf{GLT4}), from the GLT theory, we infer that the spectral symbol is given by $\{h^2X_n\}_n\sim_{\sigma,\lambda}f_0$.

Moreover, if we sample the latter eigenvalue symbol $f_0(\theta)$  with the grid $\theta_{j,n}=j\pi/(n+1)$, associated with the  {$\tau_{0,0}$}-algebra defined in Section~\ref{sec:introduction:algebras}, we will obtain an approximation of the eigenvalues of $h^2X_n$, with an error $\mathcal{O}(h)$. Instead, since $h^2X_n$ belongs to the  {$\tau_{0,1}$}-algebra, see Section~\ref{sec:introduction:algebras}, the exact eigenvalues of $h^2X_n$ are given by 
\begin{linenomath*}
  \begin{align}
  \lambda_j(h^2X_n)&=2+h^2-2\cos(\theta_{j,n}^{(0,1)}),\label{eq:exmp1:meigsymbol}\\
  \theta^{(0,1)}_{j,n}&=\frac{\pi(j-1/2)}{n+1/2},\quad j=1,\ldots,n.\nonumber
 \end{align} 
\end{linenomath*}
Sampling $f_0$ with the grid $\theta_{j,n}^{(0,1)}$ leads to an error of $h^2$ for each eigenvalue.

If we now focus on the Toeplitz momentary symbols, we deduce that the sequence $\{h^2X_n\}_n$ has Toeplitz momentary symbols given by $f_n(\theta)$
with  
\begin{linenomath*}
  \begin{align}
f_n(\theta)=2+h^{2}-2\cos\theta,\nonumber
 \end{align} 
\end{linenomath*}
in accordance with Definition \ref{def:momentarysymbols}.
If we sample the latter on the grid ${\theta^{(0,1)}_{j,n}}$, we obtain the exact eigenvalues  since the evaluations $ f_n(\theta_{j,n})$ coincide 
with \eqref{eq:exmp1:meigsymbol}.
Consequently this example highlights that, for finite matrices, the \cred{Toeplitz momentary symbols $f_{n}(\theta)$ describe more accurately the spectrum than the standard spectral symbol $f_0(\theta)$, from the theory of GLT matrix sequences, with $f_0= \lim_{n\to \infty} f_n$ uniformly on the definition domain, in accordance with Theorem \ref{moment-vs-glt-t}.}

Note that if in~\eqref{eq:exmp1:problem} pure Dirichlet boundary conditions, instead of Dirichlet-Neumann, are imposed then the matrix $h^2X_n$ belongs to the  {$\tau_{0,0}$}-algebra, since $R_n=0$. The eigenvalues $\lambda_j(h^2X_n)$ can then be computed by changing $\theta_{j,n}^{(0,1)}$ to $\theta^{(0,0)}_{j,n}$ in~\eqref{eq:exmp1:meigsymbol}; $\theta_{j,n}^{(0,0)}$ is defined in Table~\ref{tbl:taugrids}. \cred{Similarly, if periodic boundary conditions are imposed in~\eqref{eq:exmp1:problem}, the matrix $h^2X_n$ would be circulant and $\theta_{j,n}^c$, defined in~\eqref{eq:introduction:background:circulant:grid-circ}, should be used in~\eqref{eq:exmp1:meigsymbol} to obtain the exact eigenvalues. Hence, the different low-rank matrices $R_n$ in \eqref{eq:exmp1:matrix} from boundary conditions shifts the grid which gives the exact eigenvalues using~\eqref{eq:exmp1:meigsymbol}.

However, as stressed in Remark \ref{rem:wild eigs}, this result is possible since the main terms are Hermitian and hence normal. The nonnormal setting is delicate and the notions of Toeplitz generating function and Toeplitz momentary symbols could lead to wrong conclusions, due to the wild behavior of the eigenvalues.}

\end{example}

\begin{example}
\label{exmp:2}
In this example we study a constructed non-Hermitian matrix sequence where we have four different symbols: the singular and eigenvalue symbols from GLT theory, and the respective momentary symbols.
Consider
\begin{linenomath*}
  \begin{align}
X_n&=\underbrace{
\left[\begin{smallmatrix}
2&\\
1&2&\\
&\ddots&\ddots\\
&&1&2
\end{smallmatrix}\right]}_{T_n(f_1)}+\underbrace{h\mathbb{I}_n}_{N_n(f_2)}
=\left[\begin{smallmatrix}
2+h\\
1&2+h\\
&\ddots&\ddots\\
&&1&2+h
\end{smallmatrix}\right]
,\nonumber
 \end{align} 
\end{linenomath*}

where
$h=1/n$, and
\begin{linenomath*}
  \begin{align}
T_n(f_0)&=g^{(1)}(n)T_n(f_1),\nonumber\\
N_n(f_2)&=g^{(2)}(n)T_n(f_2),\nonumber\\
g^{(1)}(n)&=1,
&&f_1(\theta)=2+\E^{\mathbf{i}\theta},\nonumber\\
g^{(2)}(n)&=h,
&&f_2(\theta)=1.\nonumber
 \end{align} 
\end{linenomath*}
By the theory of GLT sequences, the singular value symbol is $f_1$, that is,
\begin{linenomath*}
  \begin{align}\label{eq:exmp2:gltsymbol}
\{X_n\}_n&\sim_{\sigma}f_1.
 \end{align} 
\end{linenomath*}
 
Using Definition~\ref{def:momentarysymbols}, the Toeplitz momentary symbols are
\begin{linenomath*}
  \begin{align}
f_n(\theta)=2+h+\E^{\mathbf{i}\theta}.\nonumber
 \end{align} 
\end{linenomath*}

Concerning the singular values of $X_n$, they are  $\sigma_j(X_n)=\sqrt{\lambda_j(X_n^\textsc{t}X_n)}$, and can be approximated by sampling $|f(\theta)|$ or $|f_n(\theta)|$ with an appropriate grid.  However, now we look at the \cred{matrix sequence $\{X_n^\textsc{t}X_n\}_n$, and by the GLT theory we infer that
\begin{linenomath*}
  \begin{align*}
\{X_n^\textsc{t}X_n\}_n&\sim_{\sigma,\lambda}\cred{f_1(-\theta)f_1(\theta)}=g(\theta)=5+4\cos\theta,\nonumber\\
\end{align*} 
\end{linenomath*}
while $\{X_n^\textsc{t}X_n\}_n$ has Toeplitz momentary symbols given by $f_n(-\theta)f_n(\theta)=g_n(\theta)=1+(2+h)^2+2(2+h)\cos\theta$.}

We know that for every $n$ the matrix $X_n^\textsc{t}X_n= {T_{n,0,-1/(2+h)}}(g_n)$, since,
\begin{linenomath*}
  \begin{align}
X_n^{\textsc{t}}X_n&=
\left[\begin{smallmatrix}
\hat{g}_{n_0}&\hat{g}_{n_1}\\
\hat{g}_{n_1}&\hat{g}_{n_0}&\hat{g}_{n_1}\\
&\ddots&\ddots&\ddots\\
&&\hat{g}_{n_1}&\hat{g}_{n_0}&\hat{g}_{n_1}\\
&&&\hat{g}_{n_1}&\hat{g}_{n_0}
\end{smallmatrix}\right]
-\frac{1}{2+h}
\left[\begin{smallmatrix}
&\\
&&\\
&&&\\
&&&&\\
&&&&\hat{g}_{n_1}
\end{smallmatrix}\right]
,\nonumber
 \end{align} 
\end{linenomath*}
where $\hat{g}_{n_0}=1+(2+h)^2$ and $\hat{g}_{n_1}=2+h$. 
 As $n$ grows, the matrix $X_n^\textsc{t}X_n$ tends towards the matrix  {$T_{n,0,-1/2}(g)$}. We have no closed form expressions for the grids $\theta_{j,n}^{(0,-1/2)}$ or $\theta_{j,n}^{(0,-1/(2+h))}$. 
\end{example} 

\cred{
\subsubsection{Analysis of the matrix sequence in Example \ref{exmp:2}} 
 
Taking into consideration the discussion in Example \ref{exmp:2},} in the following lemma{,} we provide a bound for part of the spectrum of matrices belonging to {the}  {$\tau_{0,-1/2}$}-algebra.  {The same argument can be done for matrices belonging to the  {$\tau_{0,-1/(2+h)}$}-algebra.}
\begin{lemma}
Let $f(\theta)=\hat{f}_0+2\hat{f}_1\cos\theta$ be a monotonically decreasing trigonometric polynomial. Then, for $j=2,\dots,n-1$,
$$
\lambda_j( {T_{n,0,-1}(f)})\le\lambda_j( {T_{n,0,-1/2}(f)})\le \cred{\lambda_{j+1}}( {T_{n,0,0}(f)}).
$$
\end{lemma}
 \begin{proof}
 Since $f$ is monotonically decreasing, from the relations between the grids in Table \ref{tbl:taugrids}, we have that for $j=1,\dots,n-1,$
\begin{linenomath*}
  \begin{equation}\label{eq:relation_00_0-1}
\lambda_j( {T_{n,0,-1}(f)})\le\lambda_j( {T_{n,0,0}(f)})\le \lambda_{j+1}( {T_{n,0,0}(f)}).
 \end{equation}\end{linenomath*}

We can write the matrix $ {T_{n,0,-1/2}(f)}$ in terms of rank 1 correction of the matrices $ {T_{n,0,0}(f)}$ and $ {T_{n,0,-1}(f)}$ as follows:
 \begin{linenomath*}
  \begin{align}
 {T_{n,0,-1/2}(f)}&= {T_{n,0,0}(f)}+ \left(-\frac{\hat{f}_1}{2}{\rm \mathbf{e}}_n{\rm \mathbf{e}}_n^{\textsc{t}}\right),\nonumber\\
 {T_{n,0,-1/2}(f)}&= {T_{n,0,-1}(f)}+\left(\frac{\hat{f}_1}{2}{\rm \mathbf{e}}_n{\rm \mathbf{e}}_n^{\textsc{t}}\right),\nonumber
  \end{align} 
\end{linenomath*}
where $\mathbf{e}_n=[0,0,\dots,1]^\textsc{t}$.
Hence, from the Interlacing Theorem \cite{interlacing}, for $j=1,\dots,n-1,$
 \begin{linenomath*}
  \begin{align}
\lambda_j( {T_{n,0,-1}(f)})\le \lambda_j( {T_{n,0,-1/2}(f)})\le \lambda_{j+1}( {T_{n,0,-1}(f)}),\nonumber
  \end{align} 
\end{linenomath*}
 and, for $j=2,\dots, n$,
 \begin{linenomath*}
  \begin{align}
\lambda_j( {T_{n,0,0}(f)})\le \lambda_j( {T_{n,0,-1/2}(f)})\le \lambda_{j+1}( {T_{n,0,0}(f)}).\nonumber
  \end{align} 
\end{linenomath*}
 Then, if we combine the latter relations together with formula (\ref{eq:relation_00_0-1}), we have 
 that for $j=2,\dots,n-1$,
\[
\lambda_j( {T_{n,0,-1}(f)})\le \lambda_j( {T_{n,0,-1/2}(f)})\le \lambda_{j+1}( {T_{n,0,0}(f)}).
\] 

 \end{proof}

\begin{figure}[!ht]
\centering
\includegraphics[width=0.75\textwidth]{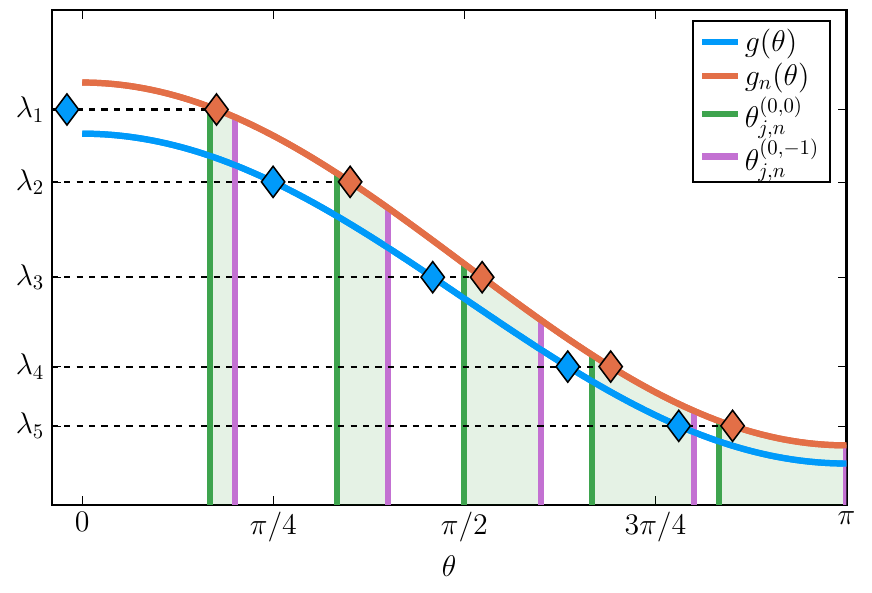}
\caption{Example 2: {The eigenvalues $\lambda_j(X^\textsc{t}_nX_n)$ for $n = 5$. The largest eigenvalue, $\lambda_1$, is an outlier for the standard symbol $g(\theta)$ (blue line), but is in the range of the momentary symbol $g_n(\theta)$ (red line). Light green regions show the intervals where the respective eigenvalues lie, bounded by $\theta_{j,n}^{(0,0)}$ and $\theta_{j,n}^{(0,-1)}$.}}
\label{fig:exmp2:grids}
\end{figure}
To illustrate the relation between the different grids, in Figure~\ref{fig:exmp2:grids} we show the spectrum of $X^\textsc{t}_nX_n$ for $n = 5$. On the {ordinate} the five eigenvalues $\lambda_j$, $j = 1,\dots,5$ are indicated. Diamonds indicate when their values are attained by the GLT symbol and by the Toeplitz momentary symbol. The upper red curve is the graph of the Toeplitz momentary symbols $g_n(\theta)$, the lower blue curve the graph of the GLT symbol $g(\theta)$. Further, vertical bars represent the grids $\theta^{(0,0)}_{j,n}$ (green) and $ {\theta^{(0,-1)}_{j,n}}$ (violet). Clearly, the true eigenvalues of $X^\textsc{t}_nX_n$ are attained by $g_M(\theta)$ in between the corresponding grid points of $\theta^{(0,0)}_{j,n}$ and $ {\theta^{(0,-1)}_{j,n}}$, i.e., in the light green area. This is not true for $g(\theta)$ from the GLT theory, further the GLT symbol cannot attain the value of $\lambda_1$ at all, since it is an outlier. {Therefore, another advantage of using the Toeplitz momentary symbols with respect to the GLT symbol is a better approximations of possible outliers.}

\cred{
Finally, a simple observation on the eigenvalues of the non-Hermitian $X_n$ follows.
By direct inspection we have
\begin{linenomath*}
  \begin{align}
\{X_n\}_n&\sim_{\lambda} 2.\nonumber
 \end{align} 
\end{linenomath*}
Note that $f_1$ in \eqref{eq:exmp2:gltsymbol} is not equal to $2$, and there is no general approach in the theory {of} GLT sequences, or elsewhere, to find the spectral symbol for general non-Hermitian matrix sequences and this because it is just impossible as emphasized in Remark \ref{rem:wild eigs}. }

\begin{example}

In this example we study a bivariate problem, from a space-time discontinuous Galerkin discretization~\cite[Example 6.2]{benedusi181}. Time is considered the first variable and the corresponding discretization parameter is $N$. The second variable is in space, discretized by the parameter $n$. \cred{Hence, as in \cite[Example 6.2]{benedusi181}, we set $\mathbf{n}=(N,n-1)$, and the resulting matrix has the form}
\begin{linenomath*}
  \begin{align}
2Nn^{-1}C_{N,n}^{[1,1,0]}(1)=
\left[\begin{array}{rrrrrr}
 {A^{[1,1,0]}_{2(n-1)}}\\
 {B^{[1,1,0]}_{2(n-1)}}& {A^{[1,1,0]}_{2(n-1)}}\\
&\ddots&\ddots\\
 &&  {B^{[1,1,0]}_{2(n-1)}}& {A^{[1,1,0]}_{2(n-1)}}
\end{array}
\right],\nonumber
 \end{align} 
\end{linenomath*}
where 
\begin{small}
\begin{linenomath*}
  \begin{align}
 {A^{[1,1,0]}_{2(n-1)}}&
=\frac{N}{12n^2}\left(\left[
\begin{array}{rrrrrrr}
9&-9\\
3&5
\end{array}
\right]\otimes T_{n-1}(2+\cos\theta_2)\right)+\left[
\begin{array}{rrrrrrr}
3&0\\
0&1
\end{array}
\right]\otimes T_{n-1}(1-\cos\theta_2),\nonumber\\
 {B^{[1,1,0]}_{2(n-1)}}&
=\frac{N}{12n^2}\left(\left[\begin{array}{rrrrrrr}
0&-12\\
0&4
\end{array}\right]\otimes T_{n-1}(2+\cos\theta_2)\right)
.\nonumber
 \end{align} 
\end{linenomath*}
\end{small}
From the structure of $2Nn^{-1}C_{N,n}^{[1,1,0]}(1)$ it is possible to find a suitable permutation matrix $P\in \mathbb{R}^{2N(n-1)\times2N(n-1)}$ such that $2Nn^{-1}C_{N,n}^{[1,1,0]}(1)$ is transformed into a $2\times 2$ block bi-level Toeplitz matrix $X_{\mathbf{n}}=P\left(2Nn^{-1}C_{N,n}^{[1,1,0]}(1)\right) P^\textsc{t}$ of the form \[X_{\mathbf{n}}=g^{(1)}({\mathbf{n}})T_\mathbf{n}(\mathbf{f}^{(1)})+g^{(2)}({\mathbf{n}})T_\mathbf{n}(\mathbf{f}^{(2)}),\]
where, $g^{(1)}({\mathbf{n}})=1$, $g^{(2)}({\mathbf{n}})=\frac{N}{n^2}$ and, following the notation in (\ref{eq:introduction:matrixvaluedsymbol}),
\begin{linenomath*}
  \begin{equation*}
\mathbf{f}^{(1)}(\theta_1,\theta_2)= \hat{\mathbf{f}}^{(1)}_{(0,0)}+\hat{\mathbf{f}}^{(1)}_{(0,1)} \E^{\mathbf{i}\theta_2}+\hat{\mathbf{f}}^{(1)}_{(0,-1)} \E^{-\mathbf{i}\theta_2},
 \end{equation*}\end{linenomath*}
\begin{linenomath*}
  \begin{equation*}
  \begin{split}
&\mathbf{f}^{(2)}(\theta_1,\theta_2)=\\ &\hat{\mathbf{f}}_{(0,0)}^{(2)}+\hat{\mathbf{f}}^{(2)}_{(0,1)} \E^{\mathbf{i}\theta_2}+ \hat{\mathbf{f}}^{(2)}_{(1,0)} \E^{\mathbf{i}\theta_1}+\hat{\mathbf{f}}_{(0,-1)}^{(2)} \E^{-\mathbf{i}\theta_2}+\hat{\mathbf{f}}^{(2)}_{(1,1)}\E^{\mathbf{i}(\theta_1+\theta_2)}+\hat{\mathbf{f}}^{(2)}_{(1,-1)}\E^{\mathbf{i}(\theta_1-\theta_2)}.
\end{split}
 \end{equation*}\end{linenomath*}
In particular, we have
	\begin{linenomath*}
  \begin{align*}\nonumber
	\hat {\mathbf{f}}^{(1)}_{(0,0)} =\begin{bmatrix}
	3 &0  \\
	\\
		0 & 	1 
	\end{bmatrix},& \nonumber \quad 
		\hat{\mathbf{f}}^{(1)}_{(0,1)}= \cred{\hat{\mathbf{f}}^{(1)}_{(0,-1)}}= \begin{bmatrix}
	-\frac{3}{2} &0\nonumber \\
	\\
		0 & -\frac{1}{2} 
	\end{bmatrix},\nonumber
	 \end{align*} 
\end{linenomath*}
	and
	\begin{linenomath*}
  \begin{align*}\nonumber
	\hat {\mathbf{f}}^{(2)}_{(0,0)} =\begin{bmatrix}
	\frac{3}{ 2} &-\frac{3}{2 }  \\
	\\
		\frac{1}{2 }   & 	\frac{5}{6 } 
	\end{bmatrix},& \nonumber \quad 
		\hat{\mathbf{f}}^{(2)}_{(0,1)}= \hat{\mathbf{f}}^{(2)}_{(0,-1)}= \begin{bmatrix}
	\frac{3}{ 8} &-\frac{3}{ 8} \nonumber \\
	\\
		\frac{1}{ 8}   & 	\frac{5}{24 } 
	\end{bmatrix},\nonumber
	\\ 
	\\
		\hat {\mathbf{f}}^{(2)}_{(1,0)} =\begin{bmatrix}
	0 &-2  \\
	\\
		0   & 	\frac{2}{3 } 
	\end{bmatrix},& \nonumber\quad 
	\hat {\mathbf{f}}^{(2)}_{(1,1)} =\cred{\hat{\mathbf{f}}^{(2)}}_{(1,-1)}=\begin{bmatrix}
	0 & -\frac{1}{2}  \\
	\\
	0&\frac{1}{6 }  
	\end{bmatrix}.
	 \end{align*} 
\end{linenomath*} 
Note that the term $g^{(2)}({\mathbf{n}})$ depends on the behavior  of $\frac{N}{n^2}$. 
 In~\cite[Example 6.2]{benedusi181} the GLT symbol is defined by assuming $N/n^2\to 0$ as $N,n\to\infty$, that is,
 \[X_\mathbf{n}=T_\mathbf{n}(\mathbf{f}^{(1)}),\]
 hence,
\begin{linenomath*}
  \begin{align}
\{X_\mathbf{n}\}_{\mathbf{n}}\sim_{\sigma}\mathbf{f}^{(1)},\nonumber
 \end{align} 
\end{linenomath*}
where we can simplify the expression of $\mathbf{f}^{(1)}$ as
\begin{linenomath*}
  \begin{align}
\mathbf{f}^{(1)}(\theta_1,\theta_2)&=\left[
\begin{array}{rrrrr}
3&0\\
0&1
\end{array}
\right]
(1-\cos\theta_2)
.\nonumber
 \end{align} 
\end{linenomath*}
 An equally valid GLT symbol would be to assume $g^{(2)}(\mathbf{n})=N/n^2=1$. In this setting the sequence is
 \[\{X_\mathbf{n}\}_{\mathbf{n}}=\{T_\mathbf{n}(\mathbf{f}^{(1)})+T_\mathbf{n}(\mathbf{f}^{(2)})\}_{\mathbf{n}}\]
 and the singular value distribution is given by 
\begin{linenomath*}
  \begin{align}
\{X_\mathbf{n}\}_{\mathbf{n}}\sim_{\sigma}\tilde{\mathbf{f}}=\mathbf{f}^{(1)}+\cred{\mathbf{f}^{(2)}},\nonumber
 \end{align} 
\end{linenomath*}
where

\begin{linenomath*}
  \begin{align*}
&\tilde{\mathbf{f}}(\theta_1,\theta_2)=\\
&\left[
\begin{array}{rrrrr}
 \frac{9}{2}  &         -\frac{3}{2}   \\ \\
    \frac{1}{2}   &   \frac{11}{6}     
\end{array}
\right] 
+\left[
\begin{array}{rrrrr}
 \frac{15}{4}    &   \frac{-3}{4}     \\ \\
  \frac{1}{4}       & \frac{17}{12}      \\
\end{array}\right]\cos\theta_2+
\left[
\begin{array}{rrrrr}
 0    &     -1     \\
 \\
  0    &      \frac{1}{3}  \\
\end{array}\right]
\cos\theta_2\E^{\mathbf{i}\theta_1}
+ \left[
\begin{array}{rrrrr}
 0    &     -2     \\ \\
  0    &      \frac{2}{3}  \\
\end{array}\right]\E^{\mathbf{i}\theta_1}
.\nonumber
 \end{align*} 
\end{linenomath*}

Note that for a diverging choice of $g^{(2)}({\mathbf{n}})$, the GLT symbol is not defined, unless we proceed to a proper scaling.

However, the momentary singular value symbol can be constructed independently from the behavior of $g^{(2)}({\mathbf{n}})$. Then, 
$\{X_{\mathbf{n}}\}_{\mathbf{n}}$ has Toeplitz momentary symbols given by $\mathbf{f}_{\mathbf{n}}$ with 
\begin{linenomath*}
  \begin{align}
&\mathbf{f}_{\mathbf{n}}(\theta_1,\theta_2)=\mathbf{f}^{(1)}(\theta_1,\theta_2)
+g^{(2)}({\mathbf{n}})\mathbf{f}^{(2)}(\theta_1,\theta_2).
\nonumber
 \end{align} 
\end{linenomath*}

The same reasoning as in Example~\ref{exmp:2} can be used for choosing a grid for attaining a good approximation of the singular values of $X_{\mathbf{n}}$. Because of the bidiagonal structure of $2Nn^{-1}C_{N,n}^{[1,1,0]}(1)$ we can, after symmetrization in both variables, define the Toeplitz momentary symbol for $X_{\mathbf{n}}$, that is, $\{X_{\mathbf{n}}\}_{\mathbf{n}}$ has Toeplitz momentary symbols defined as
\begin{linenomath*}
  \begin{align}
\boldsymbol{\mathfrak{f}}_{\mathbf{n}}(\theta_1,\theta_2)&=
\mathbf{f}^{(1)}(\theta_1,\theta_2)
+
\frac{N}{12n^2}\left[
\begin{array}{ccccc}
9&\mathbf{i}\sqrt{27}\\
\mathbf{i}\sqrt{27}&5
\end{array}
\right]
(2+\cos\theta_2)
.\label{eq:dg:momentaryeigenvaluesymbol}
 \end{align} 
\end{linenomath*}
The exact eigenvalues are given by sampling the momentary eigenvalue symbol with two grids $\theta_{j,N}^{(1)}$ and $\theta_{j,n-1}^{(2)}$.
For $\theta_{j,N}^{(1)}$ any grid can be used since the symbol \eqref{eq:dg:momentaryeigenvaluesymbol} does not explicitly depend on $\theta_1$. Furthermore, this means that the multiplicity of all distinct eigenvalues of $X_\mathbf{n}$ is $N$. This is not taken into account by the univariate symbol $\mathbf{f}_{[1,1,0]}^{[1,1]}(\theta)$  in~\cite[Example 6.2]{benedusi181},
\cred{
  \begin{equation*}
  \mathbf{f}_{[1,1,0]}^{[1,1]}(\theta)=(2-2\cos\theta)\left[\begin{array}{rr}3/2&0\\0&1/2\end{array}\right].
  \end{equation*}
}
We have the grid $\theta_{j,n-1}^{(2)}=\frac{j\pi}{n}$ for $j=1,\ldots,n-1$. For each sampling of \eqref{eq:dg:momentaryeigenvaluesymbol} a $2\times 2$ eigenvalue problem is to be solved (or an analytic expression can be derived for two separate eigenvalue functions, as it is done in~\cite[Example 6.2]{benedusi181}).
\end{example}

\section{Non-square Toeplitz matrices}
\label{sec:non-square}
\cred{For many applications and their analysis, it is often recommended or even necessary to consider non-square Toeplitz matrices: a canonical example is given by the projector and prolongation operators in multigrid algorithms (see e.g. \cite{FS2}), but we can also find such structures in the non-diagonal blocks of the two by two saddle point coefficient matrices stemming from the numerical approximation of Navier-Stokes equations (see \cite{Doro,Trava} and references therein). Furthermore, the analysis of level by level multigrid matrix sequences via the GLT theory was sketched in \cite[Section 3.7]{GLT-LAA2}}.
In this section we formalize some useful \cred{definitions}, applicable both in the standard GLT setting, and for the momentary symbols.
In \cred{Definition}~\ref{def:nonsquaresymbol} we define symbols that are matrix-valued, but not square. These symbols generate non-square Torplitz matrices, by the standard \cred{definition}. In \cred{Definition}~\ref{def:nonsquaretoeplitz2} we have the standard \cred{definition} of a non-square Toeplitz matrix, generated by scalar or square matrix-valued symbols. Combining these non-square Toeplitz matrices with standard Toeplitz matrices, we can describe a wider class of matrices $X_n$, and the associated matrix sequences $\{X_n\}_n$.
\cred{We start with a simple concrete example, in order to make the notations used in the rest of the section easier to understand.

For $f(\theta)=2-2\cos\theta$, the generating function of the standard scaled Laplacian, we have
\begin{equation*}
T_n(f)=\left[
\begin{array}{rrrrrrrrrr}
2&-1\\
-1&2&-1\\
&\ddots&\ddots&\ddots\\
&&-1&2&-1\\
&&&-1&2
\end{array}
\right].
\end{equation*}
Setting
\begin{equation*}
\mathbf{f}^{[2]}(\theta)=
\left[\begin{array}{rr}2&-1\\-1&2\end{array}\right]+
\left[\begin{array}{rr}0&-1\\0&0\end{array}\right]\E^{\mathbf{i}\theta}+
\left[\begin{array}{rr}0&0\\-1&0\end{array}\right]\E^{-\mathbf{i}\theta},
\end{equation*}
with $N=n/2$ and an even $n$, we infer $T_n(f)=T_N(\mathbf{f}^{[2]})$ where,
\begin{small}
\begin{equation*}
\begin{footnotesize}T_N(\mathbf{f}^{[2]})=\left[
\begin{array}{ccccccccccccc}
\left[\begin{array}{rr}
2&-1\\
-1&2\\
\end{array}\right]&
\left[\begin{array}{rr}
0&\phantom{-}0\\
-1&0\\
\end{array}\right]\\
\left[\begin{array}{rr}
\phantom{-}0&-1\\
0&0\\
\end{array}\right]&
\left[\begin{array}{rr}
2&-1\\
-1&2\\
\end{array}\right]&
\left[\begin{array}{rr}
0&\phantom{-}0\\
-1&0\\
\end{array}\right]\\
\\
&\ddots&\ddots&\ddots\\
\\
&&
\left[\begin{array}{rr}
\phantom{-}0&-1\\
0&0\\
\end{array}\right]&
\left[\begin{array}{rr}
2&-1\\
-1&2\\
\end{array}\right]&
\left[\begin{array}{rr}
0&\phantom{-}0\\
-1&0\\
\end{array}\right]\\
&&&\left[\begin{array}{rr}
\phantom{-}0&-1\\
0&0\\
\end{array}\right]&
\left[\begin{array}{rr}
2&-1\\
-1&2\\
\end{array}\right]
\end{array}
\right],\end{footnotesize}
\end{equation*}
and for an odd $n$, we deduce $T_n(f)=T_N(\mathbf{f}^{[2]})$ where
\begin{equation*}
T_N(\mathbf{f}^{[2]})=\left[
\begin{array}{ccccccccccccc}
\left[\begin{array}{rr}
2&-1\\
-1&2\\
\end{array}\right]&
\left[\begin{array}{rr}
0&\phantom{-}0\\
-1&0\\
\end{array}\right]\\
\left[\begin{array}{rr}
\phantom{-}0&-1\\
0&0\\
\end{array}\right]&
\left[\begin{array}{rr}
2&-1\\
-1&2\\
\end{array}\right]&
\left[\begin{array}{rr}
0&\phantom{-}0\\
-1&0\\
\end{array}\right]\\
\\
&\ddots&\ddots&\ddots\\
\\
&&
\left[\begin{array}{rr}
\phantom{-}0&-1\\
0&0\\
\end{array}\right]&
\left[\begin{array}{rr}
2&-1\\
-1&2\\
\end{array}\right]&
\left[\begin{array}{rr}
0\\
-1\\
\end{array}\right]\\
&&&\left[\begin{array}{rr}
\phantom{-}0&-1
\end{array}\right]&
\left[\begin{array}{rr}
\phantom{-}2
\end{array}\right]
\end{array}
\right].
\end{equation*}
\end{small}
That is, in the case of a matrix-valued symbol $\mathbf{f}$ generating a Toeplitz matrix $T_N(\mathbf{f})$, the parameter $N$ does not have to be an integer (but multiple of $1/s$ if $\mathbf{f}\in \mathbb{C}^{s\times s}$, since $n=sN$ is the integer-valued size of the matrix).

Hence, for the Laplacian above it is true that $\{T_n(f)\}\sim_{\lambda} f$, but it is also true that $\{T_n(f)\}\sim_{\lambda} \mathbf{f}^{[2]}$ and this non uniqueness of the symbol is not surprising and in fact it is a richness of the theory and it was discussed in~detail in \cite{ES-NLAA}. 

According to the previous case, we provide a series of definitions and an example of application.}

\begin{definition}[$f$ and the corresponding $s\times s$ matrix-valued symbol $\mathbf{f}^{[s]}$]
\label{def:matrixvaluedsymbol}
A univariate and scalar-valued generating function $f(\theta)$ has a corresponding $s\times s$ matrix-valued generating function $\mathbf{f}^{[s]}$ defined by
\begin{linenomath*}
  \begin{align}
\mathbf{f}^{[s]}(\theta)=\sum_{\ell=-\infty}^\infty \underbrace{T_s(\E^{-\mathbf{i}\ell s\theta}f(\theta))}_{\hat{\mathbf{f}}_\ell^{[s]}}\E^{\mathbf{i}\ell\theta},\label{eq:blockversionf}
 \end{align} 
\end{linenomath*}
where $\hat{\mathbf{f}}_\ell^{[s]}$ are the corresponding matrix-valued Fourier coefficients. Then,
\begin{linenomath*}
  \begin{align}
T_{ns}(f)=T_n(\mathbf{f}^{[s]}).\nonumber
 \end{align} 
\end{linenomath*}
Moreover, it is possible to extend the idea to a  multivariate $s_1 \times s_1$ matrix-valued generating function $\mathbf{f}^{[s_1]}$. Indeed, \cred{
following a similar procedure in the other level and hence in the other variable, from $\mathbf{f}^{[s_1]}$ we define} the generating function $\mathbf{f}^{[s_1s_2]}$, which is a multivariate and $s_1s_2 \times s_1s_2$ matrix-valued function. Then, we have the equivalent definition of $T_{\mathbf{n}s_2}(\mathbf{f}^{[s_1]})$  as $T_\mathbf{n}(\mathbf{f}^{[s_1s_2]})$.
\end{definition}
\begin{remark}
If the Fourier series of a generating function $f$ exists, as defined in \eqref{eq:introduction:fourierseries}, then the circulant matrix $C_n(f)$ defined in \eqref{eq:introduction:background:circulant:schur} can be rewritten as
 \begin{linenomath*}
  \begin{align}
C_n(f)&=\sum_{\ell=-\infty}^\infty T_n(\E^{\mathbf{i}\ell n\theta}f(\theta))=
\sum_{\ell=-\infty}^\infty T_n\left(\sum_{k=-\infty}^\infty\hat{f}_k\E^{\mathbf{i}(\ell n+k)\theta}\right).\label{eq:nonsquare:circulant}
 \end{align} 
\end{linenomath*}
Note that, from \eqref{eq:blockversionf} in Definition~\ref{def:matrixvaluedsymbol}, we can set $s=n$ and the equality in  \eqref{eq:nonsquare:circulant} becomes
\begin{linenomath*}
  \begin{align}
C_n(f)=\mathbf{f}^{[n]}(0)=\sum_{\ell=-\infty}^\infty \hat{\mathbf{f}}_\ell^{[n]}.\nonumber
 \end{align} 
\end{linenomath*}
 Hence, the circulant matrix $C_n(f)$ can be seen as the sum of all the Fourier coefficients of the matrix-valued version $\mathbf{f}^{[n]}(\theta)$ of $f$.
\end{remark}

\begin{definition}[Non-square matrix-valued function]
\label{def:nonsquaresymbol}
\cred{A non-square ${s\times r}\quad$ Lebesgue integrable matrix-valued function  $\mathbf{f}:[-\pi,\pi]\to\mathbb{C}^{s\times r}$, where $s,r\in\mathbb{N}$, can be defined via its Fourier coefficients $\hat{\mathbf{f}}_{{k}} \in \mathbb{C}^{s\times r}$, as follows:
\begin{linenomath*}
  \begin{align*}
\mathbf{f}({\theta})=\sum_{{k}=-\infty}^{\infty}\hat{\mathbf{f}}_{{k}}\E^{\mathbf{i} {k}{\theta}}, \quad \hat{\mathbf{f}}_{{k}} \in \mathbb{C}^{s\times r}.
\end{align*}
\end{linenomath*}
Notice that $\mathbf{f}$ Lebesgue integrable, $\mathbf{f}=\left(f_{l,m}\right)_{l=1,\ldots,r}^{m=1,\ldots,s}$, simply means that every scalar function
$f_{l,m}$ is Lebesgue integrable, $l=1,\ldots,r,\ m=1,\ldots,s$.}
 
\end{definition}

\begin{definition}[Non-square Toeplitz matrices]
\label{def:nonsquaretoeplitz}

The matrix $T_\mathbf{n}(\mathbf{f})$, with $\mathbf{n}=(n_1,\ldots,n_d)$ and  $\mathbf{f}:[-\pi,\pi]^d\to\mathbb{C}^{s\times r}$ is a multivariate and non-square matrix-valued generating function, is defined as
\begin{linenomath*}
  \begin{align*}
T_\mathbf{n}(\mathbf{f})&=
\sum_{\mathbf{k}} T_{n_1}(\E^{\mathbf{i}k_1\theta_1})\otimes \cdots \otimes
T_{n_d}(\E^{\mathbf{i}k_d\theta_1})\otimes \hat{\mathbf{f}}_{\mathbf{k}}, \quad \hat{\mathbf{f}}_{\mathbf{k}} \in \mathbb{C}^{s\times r}.
\nonumber
\end{align*} 
\end{linenomath*}

\end{definition}

In the following we want to introduce and exploit the concept of non-square identity matrix $\mathbb{I}_{n\times m}\in \mathbb{R}^{n\times m}$, $n\neq m$, that permits us to write a non-square Toeplitz matrix in terms of a square Toepitz matrix.
\begin{definition}[{Non-square identity matrix}]
\label{def:identity}
For an identity matrix $\mathbb{I}_{n\times m}\in \mathbb{R}^{n\times m}$ the following possibilities are admissible: 
\begin{enumerate}
\item $n=m$: $\mathbb{I}_{n\times m}=\mathbb{I}_n=T_n(1)$;
\item $n>m$: $\mathbb{I}_{n\times m}$ is obtained from $\mathbb{I}_{n}$ removing $(n-m)$ columns from the right;
\item $n<m$: $\mathbb{I}_{n\times m}=\mathbb{I}_{m\times n}^\textsc{t}$.
  \end{enumerate}
\end{definition}
\begin{definition}[Non-square Toeplitz matrix $T_{n\times m}(f)\in\mathbb{C}^{n\times m}$]
\label{def:nonsquaretoeplitz2}
We denote by $T_{n\times m}(f)$, with $n\neq m$,  $n,m\in\mathbb{N}$, a non-square Toeplitz matrix  generated by a univariate and scalar-valued generating function $f$. It is defined as
\begin{enumerate}
\item $n>m$: $T_{n\times m}(f)=T_{n}(f)\mathbb{I}_{n\times m}$;
\item $n<m$: $T_{n\times m}(f)=\mathbb{I}_{n\times m}T_{m}(f)$;
  \end{enumerate}
where $\mathbb{I}_{n\times m}$ is defined in Definition~\ref{def:identity}.
\end{definition}
\begin{definition}[Non-square multilevel block Toeplitz matrix $T_{\mathbf{n}\times \mathbf{m}}(\mathbf{f})$]
\label{def:nonsquaretoeplitz3}
A multilevel non-square Toeplitz matrix, denoted by $T_{\mathbf{n}\times \mathbf{m}}(\mathbf{f})$, where $\mathbf{n}=(n_1,\ldots,n_d)$ and $\mathbf{m}=(m_1,\ldots,m_d)$, generated by a multivariate and non-square matrix-valued function $\mathbf{f}:[-\pi,\pi]^d\to\mathbb{C}^{s\times r}$, where $s,r\in\mathbb{N}$ is defined as 
  \begin{linenomath*}
  \begin{align}
  T_{\mathbf{n}\times \mathbf{m}}(\mathbf{f})=\sum_{\mathbf{k}} T_{n_1\times m_1}(\E^{\mathbf{i}k_1\theta_1})\otimes \cdots\otimes T_{n_d\times m_d}(\E^{\mathbf{i}k_d\theta_d}) \otimes \hat{\mathbf{f}}_{k},\nonumber
   \end{align} 
\end{linenomath*}
where $\hat{\mathbf{f}}_{k}\in \mathbb{C}^{s\times r}$ are the Fourier coefficients of $\mathbf{f}$.

The size of the matrix $T_{\mathbf{n}\times \mathbf{m}}(\mathbf{f})$ is $d_\mathbf{n}\times d_\mathbf{m}$ which is given by $d_\mathbf{n}=sn_1n_2\cdots n_d$ and $d_\mathbf{m}=rm_1m_2\cdots m_d$.
\end{definition}

\begin{example}
\label{exmp:4}
In this example we show how a classical non-square Toeplitz matrix can be naturally treated with the aforementioned  notions of non-square generating function and related Toeplitz matrix. We consider the prolongation matrix  stemming from the linear interpolation operator used in multigrid methods (MGM) \cite{Oost, FS2}.  That is, for $n$ odd, the matrix  $P_{n\times (n-1)/2}$,  with the following structure
\begin{linenomath*}
  \begin{align}
P_{n\times (n-1)/2}&=
\left[\begin{smallmatrix}
  1\\
  2&\\
  1&  1\\
  0&  2&\\
 &  1&  1\\
 &  0 &  2&\\
  & && \ddots&\\
 &&&&  1&  1&\\
 &&&&  0&  2&\\ 
 &&&&&  1&  1\\
 &&&&&  0&  2\\
 &&&&&&  1
 \end{smallmatrix}\right]
\in\mathbb{R}^{n\times (n-1)/2}.\nonumber
 \end{align} 
\end{linenomath*}
If we consider the following $2\times 1$ matrix-valued generating function,
\begin{linenomath*}
  \begin{align*}
 {\mathbf{p}}(\theta)=\left[\begin{array}{c} 1\\ 2\end{array}\right]+\left[\begin{array}{c} 1\\ 0\end{array}\right]\E^{\mathbf{i}\theta},
  \end{align*} 
\end{linenomath*}
we can write
\begin{linenomath*}
  \begin{align}
T_{(n+1)/2}( {\mathbf{p}})&=
\left[
\begin{smallmatrix}
 1&&&&&&&  \\
2&&&&&&&  \\
 1&1&&&&&&  \\
0&2&&&&&&  \\
 &1& 1&&&&&  \\
 & 0 & 2&&&&&  \\
  & && \ddots&&&&  \\
 &&&&  1&  1&&  \\
 &&&&  0&  2&&  \\ 
 &&&&& 1& 1&  \\
 &&&&&  0&2&  \\
 &&&&&&  1&  1\\
   &  &  &  &  &  &  0&  2
\end{smallmatrix}
\right]=\left[
\begin{array}{ccc|c}
&&&0\\
&P_{n\times (n-1)/2}&&\vdots\\
&&&0\\
&&&1\\
\hline
0&\cdots&0&2
\end{array}
\right]\in \mathbb{R}^{(n+1)\times (n+1)/2}.\nonumber
 \end{align} 
\end{linenomath*}
Then, removing the last row (by multiplication from the left with $\mathbb{I}_{n\times (n+1)}$) and last column (by multiplication from the right with $\mathbb{I}_{(n+1)/2\times (n-1)/2}$), we can express the matrix $P_{n\times (n-1)/2}$ as
\begin{linenomath*}
  \begin{align}
P_{n\times (n-1)/2}=\mathbb{I}_{n\times (n+1)}T_{(n+1)/2}( {\mathbf{p}})\mathbb{I}_{(n+1)/2\times (n-1)/2}.\nonumber
 \end{align} 
\end{linenomath*}
 This implies that $P_{n\times (n-1)/2}$  shares the same momentary singular value symbol 
\begin{linenomath*}
  \begin{align}
 {\mathbf{p}}(\theta)=\left[\begin{array}{c} 1\\ 2\end{array}\right]+\left[\begin{array}{c} 1\\ 0\end{array}\right]\E^{\mathbf{i}\theta},\label{eq:prolongationmomentarysymbol}
 \end{align} 
\end{linenomath*}
with \cred{the matrix $T_{(n+1)/2}( {\mathbf{p}})+R_1$, which differs from $T_{(n+1)/2}({\mathbf{p}})$ just for a rank 1 correction matrix $R_1$, whose expression is given by 
\[
R_1=
\left[
\begin{array}{ccc|c}
&&&0\\
& &&\vdots\\
&&&0\\
&&&1\\
\hline
0&\cdots&0&2
\end{array}
\right] =(e_n+e_{n-1})e_n^T,
\]
$e_j$, $j=1,\ldots,n$, being the vectors of the canonical basis of $\mathbb{C}^n$.}

 {An additional confirmation of this fact can be seen following a more classical construction of the matrix $P_{n\times (n-1)/2}$, which can be derived in analogous way, see \cite{donatelli111}. Indeed, we can obtain $P_{n\times (n-1)/2}$} multiplying the matrix $T_n(g)$, where $g(\theta)=2+2\cos\theta$, with a so-called cutting matrix $Z_{n\times (n-1)/2}$, as follows, 
\begin{linenomath*}
  \begin{align}
P_{n\times (n-1)/2}&=T_n(g)Z_{n\times (n-1)/2},\nonumber
 \end{align} 
\end{linenomath*}
where, defining the generating function $\mathbf{f}_z(\theta)=
\left[\begin{smallmatrix}
0\\
1
\end{smallmatrix}\right]$,
we have
\begin{linenomath*}
  \begin{align}
Z_{n\times (n-1)/2}=\mathbb{I}_{n\times n+1}T_{(n+1)/2}(\mathbf{f}_z)\mathbb{I}_{(n+1)/2\times (n-1)/2}.\nonumber
 \end{align} 
\end{linenomath*}
By Definition~\ref{def:matrixvaluedsymbol}, for $s=2$, the matrix-valued version of $g$ is
\begin{linenomath*}
  \begin{align}
\mathbf{g}^{[2]}(\theta)&=T_2(g)+T_2(\E^{-2\mathbf{i}\theta}g)\E^{\mathbf{i}\theta}+T_2(\E^{2\mathbf{i}\theta}g)\E^{-\mathbf{i}\theta}\nonumber\\
&=\begin{bmatrix}
2& 1\\
1&2
\end{bmatrix}
+
\begin{bmatrix}
0& 1\\
0&0
\end{bmatrix}\E^{\mathbf{i}\theta}
+
\begin{bmatrix}
0& 0\\
1&0
\end{bmatrix}\E^{-\mathbf{i}\theta}.\nonumber
 \end{align} 
\end{linenomath*}
We then have
\begin{linenomath*}
  \begin{align}
P_{n\times (n-1)/2}=\mathbb{I}_{n\times n+1}T_{(n+1)/2}(\mathbf{g}^{[2]}\mathbf{f}_z)\mathbb{I}_{(n+1)/2\times (n-1)/2},\nonumber
 \end{align} 
\end{linenomath*}
where
\begin{linenomath*}
  \begin{align}\label{eq:project_mult}
\mathbf{g}^{[2]}(\theta)\mathbf{f}_z(\theta)&=\left(\begin{bmatrix}
2& 1\\
1&2
\end{bmatrix}
+
\begin{bmatrix}
0& 1\\
0&0
\end{bmatrix}\E^{\mathbf{i}\theta}
+
\begin{bmatrix}
0& 0\\
1&0
\end{bmatrix}\E^{-\mathbf{i}\theta}\right)\begin{bmatrix}
0\\
1
\end{bmatrix}
=\begin{bmatrix}
1\\
2
\end{bmatrix}
+
\begin{bmatrix}
1\\
0
\end{bmatrix}\E^{\mathbf{i}\theta}\\
&=\mathbf{p}_{n}(\theta) \nonumber,
 \end{align} 
\end{linenomath*}
where $\mathbf{p}_{n}(\theta)$ is defined in \eqref{eq:prolongationmomentarysymbol}.
 {Then, the first part of the example } shows that we can treat non-square ($s\neq r$) matrix-valued generating function as any other generating function, as long as we take care to transform all involved generating functions to blocks of correct sizes and scalar-valued generating functions (which are not just a constant) should be treated as matrices of size $1\times 1$ and have to be resized for valid multiplication.  

 {In the following we want to show how non-square sequences can be studied exploiting the concept of non-square momentary symbols.}

Let us consider the matrix  $h^2X_n$ defined by \ref{eq:exmp1:matrix} in the Example \ref{exmp:1} and its associated momentary symbols
\begin{linenomath*}
  \begin{equation*}
f_n(\theta)=2+h^2-2\cos\theta,
 \end{equation*}\end{linenomath*}
where $h=1/(n+1)$.
 {In many applications the  study of the spectrum of a matrix of the form 
\begin{linenomath*}
  \begin{equation*}
Y_{(n-1)/2}=P^\textsc{H}h^2X_nP,
 \end{equation*}\end{linenomath*}
could be of interest, where $P=P_{n\times(n-1)/2}$. }
 {Indeed, the matrix $Y_{(n-1)/2}$ could be seen as the matrix on the coarse level of a multigrid procedure, obtained using as prolongation operator the matrix $P$.}

 {The matrix $Y_{(n-1)/2}$ is symmetric by construction and its} resulting  {eigenvalue} momentary symbol  {can be constructed as}  
\begin{linenomath*}
  \begin{equation*}
\begin{split}
&y_n(\theta)= {\mathbf{p}^\textsc{H}_n}(\theta)\mathbf{f}_n(\theta) {\mathbf{p}_n}(\theta)=\\
&\left(
\left[\begin{smallmatrix}
1&2\\
\end{smallmatrix}\right]+\left[\begin{smallmatrix}
1&0\\
\end{smallmatrix}\right]e^{-\mathbf{i}\theta}
\right)
\left(
\left[\begin{smallmatrix}
2+h^2&-1\\
-1&2+h^2
\end{smallmatrix}\right]+
\left[\begin{smallmatrix}
0&-1\\
0&0
\end{smallmatrix}\right]e^{\mathbf{i}\theta}+
\left[\begin{smallmatrix}
0&0\\
-1&0
\end{smallmatrix}\right]e^{-\mathbf{i}\theta}
\right)
\left(
\left[\begin{smallmatrix}
1\\
2
\end{smallmatrix}\right]+
\left[
\begin{smallmatrix}
1\\
0
\end{smallmatrix}\right]e^{\mathbf{i}\theta}
\right)=\\
&4+6h^2+2(h^2-2)\cos\theta= {4-4\cos\theta+h^2(6+4\cos\theta)},
\end{split}
 \end{equation*}\end{linenomath*}
where $ {\mathbf{p}^\textsc{H}_n}(\theta)$  {is the momentary singular value symbol of $P^\textsc{H}$} and $\mathbf{f}_n(\theta)$ is the $2\times 2$ block version of $f_n(\theta)$. 
 \cred{An additional confirmation of this fact can be seen, by noticing that, by direct computation, we have}
\begin{linenomath*}
  \begin{equation*}
Y_{(n-1)/2}=T_{(n-1)/2}(y_n)+R_{(n-1)/2},
 \end{equation*}\end{linenomath*}
where $R_{(n-1)/2}$ is a matrix with the only non-zero element being a $-1$ in the bottom right corner.

 {In addition, the matrix} $Y_{(n-1)/2}$ belongs to the  {$\tau_{0,1/(2-h^2)}$}-algebra and we can employ the strategy of Example \ref{exmp:2} to choose the appropriate grid for the eigenvalue approximations  {via its momentary eigenvalue symbol}.

 \cred{Furthermore, we mention that the procedure described in the present example generalizes and justifies the approach presented in \cite{Huckle_compact}. Indeed, the author constructs the symbol at the coarse levels by the $2\times 2$ matrix-valued version of the symbol of the problem and projects it by the function $ B(x)\begin{bmatrix}
 1\\
 0
 \end{bmatrix}$, where $B(x)$ is the chosen symbol of the prolongation operator. The latter is then a particular case of the product of the form (\ref{eq:project_mult}).  Finally, we remark that in a pure GLT context the present reasoning was already considered and described concisely in \cite[Section 3.7]{GLT-LAA2}.}

\end{example}

\section{Conclusions}
\cred{
In this paper we introduced and exploited the concept of the Toeplitz momentary symbols. We showed how the idea behind its construction is similar to that of the symbol stemming from the GLT theory, but in practice it is applicable in order to obtain more precise estimates of eigenvalues and singular values. 

 We illustrated the efficacy of the momentary symbols in  Examples 1-4, including the multilevel block and non-square settings.}
{
Object of further research will be the extension of the proposed tools to more challenging structures coming from applications of interest. In particular, we plan to apply the Toeplitz momentary symbols approach to the iteration matrix sequences stemming from Parallel-in-Time problems.
}

{
Finally, we mention that {in many recent works \cite{matrix_less_block, matrix_less}, under specific hypotheses on the generating function $f$, it is possible to give an accurate description of the eigenvalues of $T_n(f)$ via an asymptotic expansion of the form}
\begin{linenomath*}
  \begin{align}
\lambda_j(T_n(f))=c_0(\theta_{j,n})+hc_1(\theta_{j,n})+h^2c_2(\theta_{j,n})+h^3c_3(\theta_{j,n})+\ldots,\nonumber
 \end{align} 
\end{linenomath*} and the functions $c_k(\theta)$ can be approximated by so-called matrix-less methods. We highlight that in the Hermitian case, we have $c_0=f$ and the subsequent functions $c_1,c_2,\ldots$ can be seen as part of the momentary singular value symbol $f_n$. In the non-Hermitian case, \cred{the situation is much more involved and the approach can be successful only in specific well selected cases, which deserve a careful study}. Then, efficient and fast algorithms  can be designed for computing the singular values and eigenvalues of $T_n(f)$ (plus its possible block, and variable coefficient generalizations) and this will be investigated in the future.
}

\section*{Acknowledgments}

 {We are thankful to Dr. Carlo Garoni for the insightful discussions and suggestions.}
This work was partially supported by ``Gruppo Nazionale per il Calcolo Scientifico'' (GNCS-INdAM).  {The second author was partially funded by the Swedish Research Council through the
International Postdoc Grant (Registration Number 2019-00495).}

\bibliography{biblio}

\end{document}